\documentclass[11pt,reqno]{amsart}

\usepackage[T1]{fontenc}
\usepackage{lmodern}
\usepackage{microtype}
\usepackage{amsmath,amssymb,mathtools,mathrsfs}
\usepackage{booktabs}
\usepackage{array}
\usepackage{enumitem}
\usepackage{xcolor}
\usepackage[hidelinks]{hyperref}
\hypersetup{
  pdftitle={Orthogonal-polynomial models and numerical invariants of principal bidisc submodules and quotient modules},
  pdfauthor={Yufeng Lu and Chao Zu},
  pdfsubject={Principal submodules and quotient modules of the Hardy module over the bidisc},
  pdfkeywords={Hardy module, bidisc, orthogonal polynomials, Verblunsky coefficients, CMV matrices, core operator, quotient module}
}

\allowdisplaybreaks
\numberwithin{equation}{section}

\newtheorem{theorem}{Theorem}[section]
\newtheorem{proposition}[theorem]{Proposition}
\newtheorem{lemma}[theorem]{Lemma}
\newtheorem{corollary}[theorem]{Corollary}

\newtheorem{problem}[theorem]{Problem}
\theoremstyle{definition}

\newtheorem{example}[theorem]{Example}
\theoremstyle{remark}
\newtheorem{remark}[theorem]{Remark}

\newcommand{\D}{\mathbb D}
\newcommand{\T}{\mathbb T}
\newcommand{\C}{\mathbb C}
\newcommand{\N}{\mathbb N}
\newcommand{\Nzero}{\mathbb N_0}
\newcommand{\HH}{H^2(\D^2)}
\newcommand{\calM}{\mathcal M}
\newcommand{\calP}{\mathcal P}

\newcommand{\Ran}{\operatorname{Ran}}

\newcommand{\spec}{\operatorname{spec}}

\newcommand{\Span}{\operatorname{span}}

\newcommand{\HS}{\mathrm{HS}}
\newcommand{\Sr}{\mathcal S}
\newcommand{\dd}{\,\mathrm d}
\newcommand{\one}{\mathbf 1}
\newcommand{\eps}{\varepsilon}

\title[Verblunsky coefficients and bidisc numerical invariants]
{Verblunsky Coefficients, CMV Matrices, and Numerical Invariants of Homogeneous Principal Bidisc Submodules}

\author{Yufeng Lu}
\address{School of Mathematical Sciences, Dalian University of Technology, Dalian 116024, China}
\email{lyfdlut@dlut.edu.cn}

\author{Chao Zu}
\address{School of Mathematical Sciences, Dalian University of Technology, Dalian 116024, China}
\email{zuchao@dlut.edu.cn}


\begin{document}

\begin{abstract}
Let $[p]$ be the principal submodule generated by a polynomial in
$H^2(\mathbb D^2)$.  For homogeneous $p$, the homogeneous slices of $[p]$
admit a weighted OPUC model in which the two wandering vectors are an
orthonormal polynomial and its reversal.  We show that the associated
Verblunsky coefficients determine the singular values of the
wandering-projection product and the restricted cross-commutator, as well as
the non-zero spectrum of the core operator.  Toeplitz-determinant and
Mahler-measure identities yield exact Fredholm determinants and Schatten
estimates, while $[(z-w)^N]$ rules out a uniform Hilbert--Schmidt bound.  The
same model gives explicit singular values of $[S_z^*,S_w]$ on the homogeneous
quotient $H^2(\mathbb D^2)\ominus[p]$; for $p=(z-w)^N$, its squared
Hilbert--Schmidt norm is asymptotic to $N$.  For arbitrary polynomial
generators, we construct a weighted bivariate model with a doubly Toeplitz,
block-banded moment matrix and prove
$C_p^2|_{\mathscr E_z}=\Gamma_p^*\Gamma_p$, relating the core spectrum to the
cross-Gram operator between the two edge spaces.  We also discuss
cyclic-factor obstructions, represent higher numerical invariants by
alternating CMV products, and give a quadratic counterexample to their
proposed monotonicity.
\end{abstract}

\maketitle

\section{Introduction}

The Hardy space $\HH$ is naturally a Hilbert module over $\C[z,w]$.
Its closed submodules are invariant under multiplication by both coordinate
functions, but, in contrast with the one-variable situation governed by
Beurling's theorem, they do not admit a general inner-function
classification.  The defect spaces of the two restricted coordinate shifts,
their cross-commutator, and the associated core operator therefore play an
important role in the structure theory of bidisc submodules; see
\cite{ChenGuo,GuoYang,Rudin,Yang1999,Yang2001,Yang2004,Yang2005,YangSurvey}.

Let $\calM\subset\HH$ be a submodule, let $R_z=M_z|_{\calM}$ and
$R_w=M_w|_{\calM}$, and put
\[
 P_z=I-R_zR_z^*=P_{\calM\ominus z\calM},\qquad
 P_w=I-R_wR_w^*=P_{\calM\ominus w\calM}.
\]
The cross-commutator and core operator are
\begin{align*}
 X_{\calM}&=[R_z^*,R_w],\\
 C_{\calM}&=I-R_zR_z^*-R_wR_w^*+R_zR_wR_z^*R_w^*.
\end{align*}
The submodule is called Hilbert--Schmidt when $C_{\calM}$ is
Hilbert--Schmidt.  In that case the two basic numerical invariants are
\[
 \Sigma_0(\calM)=\|P_zP_w\|_{\HS}^2,
 \qquad
 \Sigma_1(\calM)=\|X_{\calM}\|_{\HS}^2,
\]
and they satisfy $\Sigma_0-\Sigma_1=1$; see
\cite{AzariKeyLuYang,Yang1999,Yang2001,Yang2005}.

For a homogeneous polynomial generator, Azari Key, Lu and Yang
\cite{AzariKeyLuYang} constructed the wandering vectors by cofactors of
finite Toeplitz Gram matrices and expressed the core spectrum through those
determinants.  Our starting point is that the entire construction belongs
naturally to the classical theory of orthogonal polynomials on the unit
circle (OPUC).  If $p$ is homogeneous and $q(\zeta)=p(\zeta,1)$, then
multiplication by $p$ identifies the $n$th homogeneous slice of $[p]$ with
the polynomials of degree at most $n$ in
\[
 L^2(\mu_p),\qquad
 \dd\mu_p(\zeta)=\frac{|q(\zeta)|^2}{\|q\|_2^2}\,\dd m(\zeta).
\]
Under this unitary identification, the two slice-wise wandering vectors are
the orthonormal polynomial $\varphi_n$ and its reversed polynomial
$\varphi_n^*$.  Their relative position is measured exactly by the
Verblunsky coefficient $\alpha_{n-1}$.

The first main theorem, obtained in Sections~\ref{sec:model}--\ref{sec:operators},
is the following dictionary.  The Schatten class $\Sr_r$ has its usual
quasi-norm interpretation when $0<r<1$.

\begin{theorem}[OPUC dictionary]\label{thm:intro-dictionary}
Let $p\neq0$ be homogeneous, let $\calM=[p]$, and let $(\alpha_n)_{n\geq0}$
be the Verblunsky coefficients of $\mu_p$.  Then:
\begin{enumerate}[label=\textup{(\roman*)}]
\item the non-zero singular values of $P_zP_w$ are
      $1,|\alpha_0|,|\alpha_1|,\ldots$;
\item the non-zero singular values of $X_{\calM}$ are
      $|\alpha_0|,|\alpha_1|,\ldots$;
\item the non-zero eigenvalues of $C_{\calM}$ are
      $1$ and $\pm|\alpha_n|$, $n\geq0$;
\item for every $r>0$,
\[
 C_{\calM}\in\Sr_r\quad\Longleftrightarrow\quad
 (\alpha_n)\in\ell^r,
 \qquad
 \|C_{\calM}\|_{\Sr_r}^r=1+2\sum_{n\geq0}|\alpha_n|^r.
\]
\end{enumerate}
In particular, every homogeneous principal polynomial submodule is
Hilbert--Schmidt and
\[
 \Sigma_0([p])=1+\sum_{n\geq0}|\alpha_n|^2,
 \qquad
 \Sigma_1([p])=\sum_{n\geq0}|\alpha_n|^2.
\]
\end{theorem}

This reformulation has two consequences which are difficult to see from
cofactor formulas alone.  First, Szeg\H{o}'s theorem gives an exact global
identity.  If
\[
 \mathfrak M(q)=\exp\!\left(\int_{\T}\log|q|\,\dd m\right)
\]
is the Mahler measure, and $Q$ is the projection onto the common line of the
two wandering spaces, then
\begin{equation}\label{eq:intro-fredholm}
 \det\!\left(I-(P_zP_w-Q)^*(P_zP_w-Q)\right)
 =\det(I-X_{\calM}^*X_{\calM})
 =\frac{\mathfrak M(q)^2}{\|q\|_2^2}.
\end{equation}
Combining this identity with an elementary sharp comparison between the
coefficient norm and Mahler measure yields, for $d=\deg q$,
\begin{equation}\label{eq:intro-degree-bound}
 \|X_{\calM}\|_{\HS}^2
 \leq \log\frac{\|q\|_2^2}{\mathfrak M(q)^2}
 \leq\log\binom{2d}{d}<2d\log2.
\end{equation}

Second, the pure Fisher--Hartwig weight $|1-\zeta|^{2N}$ has
\[
 \alpha_n=-\frac{N}{n+N+1}.
\]
It follows that the core operator of $[(z-w)^N]$ satisfies
\begin{equation}\label{eq:intro-unbounded}
 \|C_{[(z-w)^N]}\|_{\HS}^2
 =1+2N^2\sum_{j=N+1}^{\infty}\frac1{j^2}
 =2N+\frac1{3N}-\frac1{15N^3}+O(N^{-5}).
\end{equation}
Thus the core Hilbert--Schmidt norms are unbounded even among homogeneous
principal polynomial submodules.  This gives a negative answer to Problem~10
in Yang's survey \cite{YangSurvey}.

The same OPUC data also control the complementary quotient module, but through
a different operator-theoretic mechanism.  Let $p$ have total degree $D\geq1$,
write
\[
 q(\zeta)=p(\zeta,1)=\sum_{j=0}^Dq_j\zeta^j,
 \qquad
 a=\frac{|q_0|^2}{\|q\|_2^2},\quad
 b=\frac{|q_D|^2}{\|q\|_2^2},
\]
where $q_0$ or $q_D$ is allowed to vanish, and denote by $\kappa_n$ the
positive leading coefficient of $\varphi_n$.  On
\[
 \mathcal Q_p=\HH\ominus[p]
\]
let $S_z$ and $S_w$ be the compressed coordinate shifts.  In
Section~\ref{sec:quotient} we prove that the non-zero singular values of
$Y_p=[S_z^*,S_w]$, before decreasing rearrangement, are
\begin{align}
 s_0(Y_p)&=\sqrt{(1-a)(1-b)},\label{eq:intro-quotient-s0}\\
 s_n(Y_p)&=|\alpha_{n-1}|
 \sqrt{(1-a\kappa_n^2)(1-b\kappa_n^2)},\qquad n\geq1.
 \label{eq:intro-quotient-sn}
\end{align}
The extra factors have a simple meaning: they measure how much of the two
wandering vectors survives one backward coordinate shift.  Thus the quotient
formula is not obtained merely by replacing the restricted shifts by their
compressions.  It arises from the passage from $AB$ to $BA$ in the block
matrix of the doubly commuting ambient shifts.

In particular, every such quotient cross-commutator is Hilbert--Schmidt and
\begin{equation}\label{eq:intro-quotient-bound}
 \|Y_p\|_{\HS}^2
 \leq 1+\log\frac{\|q\|_2^2}{\mathfrak M(q)^2}
 \leq1+\log\binom{2D}{D}.
\end{equation}
This degree-dependent estimate cannot be replaced by an absolute constant.
For $p_N=(z-w)^N$ we obtain an exact singular-value formula and
\begin{equation}\label{eq:intro-quotient-asymptotic}
 \|[S_z^*,S_w]\|_{\HS,\,\mathcal Q_{p_N}}^2\sim N.
\end{equation}
For comparison, the case $N=1$ recovers the self-commutator of the Bergman
shift and gives the exact value $\pi^2/3-3$.

The OPUC model also organizes the higher numerical invariants introduced in
\cite{LiuLuZu}.  Let $\mathcal C=\mathcal L\mathcal M$ be the CMV matrix of
$\mu_p$, with the standard even and odd block factors.  In
Section~\ref{sec:cmv} we prove
\begin{equation}\label{eq:intro-cmv}
 \Sigma_k([p])=
 \sum_{m\geq0}|(\mathcal C^{k-1}\mathcal L)_{2m,2m}|^2
 +\sum_{m\geq0}|(\mathcal M\mathcal C^{k-1})_{2m+1,2m+1}|^2,
 \qquad k\geq1.
\end{equation}
The formula shows that $\Sigma_k$ is the $\ell^2$-energy of a finite-range
nonlinear transform of the Verblunsky sequence.  It also makes clear that
monotonicity in $k$ is not a formal consequence of CMV unitarity: different
values of $k$ select diagonal entries of different unitary words.

This obstruction is genuine.  For
\[
 p_*(z,w)=z^2-\sqrt3zw+w^2
          =(z-e^{i\pi/6}w)(z-e^{-i\pi/6}w),
\]
the Verblunsky coefficients have an explicit six-step form.  Exact CMV
calculation, grouped into six-term blocks, gives
\begin{equation}\label{eq:intro-counterexample}
 \Sigma_4([p_*])-\Sigma_3([p_*])>\frac{13}{75000}>0.
\end{equation}
This disproves the proposed decrease of $(\Sigma_k)$, even for a quadratic
homogeneous generator.  The proof is entirely exact; the algebraic
certificate is recorded in Appendix~\ref{app:certificate}.

Finally, Section~\ref{sec:quasi} treats an $(s,t)$-quasi-homogeneous
polynomial.  Writing
\[
 p(z,w)=z^aw^bq(z^t,w^s),\qquad (s,t)=1,
\]
we show that the two wandering spaces split into $s$ and $t$ residue classes.
All mixed pairings vanish except in their common zero class.  Hence the
non-zero singular values of $P_zP_w$ are exactly those for the ordinary
homogeneous polynomial $q$; the weights and the monomial factor add only
zero blocks.  This explains why the quasi-homogeneous case is a formal
extension of the homogeneous one rather than a separate source of spectral
data.

Section~\ref{sec:general} then explains what survives for an arbitrary
polynomial generator.  If $p=p_a+\cdots+p_b$ is its decomposition into
non-zero extreme homogeneous parts, multiplication by $p$ gives a unitary
model on the polynomial closure in
\[
 L^2\!\left(\frac{|p|^2}{\|p\|^2}\,\dd m_{\T^2}\right).
\]
In this model total-degree layers at distance greater than $b-a$ are
orthogonal, so the moment matrix is block banded.  More sharply, all distinct
total-degree layers are orthogonal if and only if $a=b$.  Thus the scalar OPUC
model is not merely one convenient choice: it is a rigid consequence of
homogeneity.  The two edge spaces of this bivariate model carry a canonical
cross-Gram operator $\Gamma_p$.  We prove the exact identity
\begin{equation}\label{eq:intro-general-core}
 C_{[p]}^2\big|_{[p]\ominus z[p]}
 \simeq \Gamma_p^*\Gamma_p,
\end{equation}
so the non-extreme core eigenvalues are the positive and negative singular
values of $\Gamma_p$.  This is the precise sense in which the core spectrum
of a general polynomial principal submodule is governed by bivariate
orthogonal polynomials.  We also show that multiplication of a generator by
a cyclic polynomial leaves the principal submodule unchanged.  The elementary
example $p(z,w)=1+\lambda z$, $0<|\lambda|<1$, then proves that neither the
slice $p(\zeta,1)$ nor the Mahler-measure identity
\eqref{eq:intro-fredholm} can be used unchanged in the non-homogeneous
setting.  Section~\ref{sec:problems} formulates the resulting block-model,
Schatten, determinant, quotient and higher-invariant problems in a precise
form.

Throughout, inner products are linear in the first variable, and $m$ denotes
normalized Haar measure on $\T$ or $\T^2$, according to context.

\section{The homogeneous OPUC model}\label{sec:model}

For $n\geq0$ let
\[
 H_n=\Span\{z^jw^{n-j}:0\leq j\leq n\},
 \qquad H_{-1}=\{0\}.
\]
Then $\HH=\bigoplus_{n\geq0}H_n$.  Let
\begin{equation}\label{eq:pq-def}
 p(z,w)=\sum_{j=0}^d c_jz^jw^{d-j}\neq0,
 \qquad q(\zeta)=p(\zeta,1)=\sum_{j=0}^d c_j\zeta^j,
\end{equation}
and set $\calM=[p]$.  Since multiplication by $p$ shifts total degree by $d$,
\begin{equation}\label{eq:homogeneous-decomp}
 \calM=\bigoplus_{n\geq0}pH_n,
 \qquad
 z\calM=\bigoplus_{n\geq1}zpH_{n-1},
 \qquad
 w\calM=\bigoplus_{n\geq1}wpH_{n-1}.
\end{equation}
Multiplication by the non-zero polynomial $p$ is injective on $\C[z,w]$.
Consequently, for $n\geq1$,
\[
 \dim pH_n=n+1,
 \qquad
 \dim zpH_{n-1}=\dim wpH_{n-1}=n.
\]
It follows that
\begin{align}
 \calM\ominus z\calM
 &=\bigoplus_{n\geq0}(pH_n\ominus zpH_{n-1}),\label{eq:z-wandering-decomp}\\
 \calM\ominus w\calM
 &=\bigoplus_{n\geq0}(pH_n\ominus wpH_{n-1}),\label{eq:w-wandering-decomp}
\end{align}
and every displayed summand is one-dimensional.

Normalize the measure associated with $q$ by
\begin{equation}\label{eq:mu-def}
 \dd\mu(\zeta)=\frac{|q(\zeta)|^2}{\|q\|_2^2}\,\dd m(\zeta).
\end{equation}
Let $\calP_n$ be the polynomials of degree at most $n$ in $L^2(\mu)$.

\begin{proposition}[Unitary slice model]\label{prop:slice}
For $n\geq0$ define
\begin{equation}\label{eq:Un}
 U_nf(z,w)=\frac{p(z,w)}{\|p\|}\,w^nf(z/w),
 \qquad f\in\calP_n.
\end{equation}
Then $U_n$ is unitary from $\calP_n$ onto $pH_n$, and
\[
 U_n(\zeta\calP_{n-1})=zpH_{n-1},
 \qquad
 U_n(\calP_{n-1})=wpH_{n-1}.
\]
\end{proposition}

\begin{proof}
The expression $w^nf(z/w)$ is the homogenization of $f$ and belongs to
$H_n$.  On $\T^2$ one has $p(z,w)=w^dq(z/w)$.  The change of variables
$(z,w)=(\zeta w,w)$ preserves normalized Haar measure, and hence
\begin{align*}
 \langle U_nf,U_ng\rangle_{\HH}
 &=\frac1{\|p\|^2}\int_{\T^2}|p(z,w)|^2
       f(z/w)\overline{g(z/w)}\,\dd m(z,w)\\
 &=\frac1{\|q\|_2^2}\int_{\T}f(\zeta)\overline{g(\zeta)}
       |q(\zeta)|^2\,\dd m(\zeta)
 =\langle f,g\rangle_{L^2(\mu)}.
\end{align*}
Here $\|p\|_{\HH}=\|q\|_{L^2(\T)}$.  The monomial identities
\[
 U_n(\zeta^{j+1})=\frac{p}{\|p\|}z^{j+1}w^{n-j-1},
 \qquad
 U_n(\zeta^j)=\frac{p}{\|p\|}z^jw^{n-j}
\]
give the two subspace correspondences.
\end{proof}

Let $(\varphi_n)_{n\geq0}$ be the orthonormal polynomials for $\mu$, with
positive leading coefficients $\kappa_n$.  If $h$ has degree at most $n$,
write
\[
 h^*(\zeta)=\zeta^n\overline{h(1/\overline\zeta)}.
\]
Then $\varphi_n^*$ is a unit vector orthogonal to
$\zeta\calP_{n-1}$.  Let $(\alpha_n)_{n\geq0}\subset\D$ be the
Verblunsky coefficients and put $\rho_n=(1-|\alpha_n|^2)^{1/2}$.  We use
the standard Szeg\H{o} recursions
\begin{align}
 \rho_n\varphi_{n+1}(\zeta)
 &=\zeta\varphi_n(\zeta)-\overline{\alpha_n}\varphi_n^*(\zeta),
 \label{eq:szego}\\
 \rho_n\varphi_{n+1}^*(\zeta)
 &=\varphi_n^*(\zeta)-\alpha_n\zeta\varphi_n(\zeta).
 \label{eq:szego-reversed}
\end{align}
The following two overlaps will be used repeatedly.

\begin{lemma}[OPUC overlaps]\label{lem:overlaps}
For $n\geq0$,
\begin{align}
 \langle\varphi_{n+1}^*,\varphi_{n+1}\rangle_{\mu}
 &=-\alpha_n,\label{eq:overlap-same}\\
 \langle\varphi_n^*,\zeta\varphi_n\rangle_{\mu}
 &=\alpha_n.\label{eq:overlap-shifted}
\end{align}
\end{lemma}

\begin{proof}
Take the inner product of \eqref{eq:szego-reversed} with
$\varphi_{n+1}$.  The first term on the right is orthogonal to
$\varphi_{n+1}$, while \eqref{eq:szego} gives
$\langle\zeta\varphi_n,\varphi_{n+1}\rangle=\rho_n$.  This proves
\eqref{eq:overlap-same}.  Rewriting \eqref{eq:szego-reversed} as
\[
 \varphi_n^*=\rho_n\varphi_{n+1}^*+\alpha_n\zeta\varphi_n
\]
and using $\varphi_{n+1}^*\perp\zeta\calP_n$ proves
\eqref{eq:overlap-shifted}.
\end{proof}

\begin{theorem}[Canonical wandering bases]\label{thm:wandering-bases}
For $n\geq0$ put
\begin{equation}\label{eq:ef-def}
 e_n=U_n\varphi_n^*,\qquad f_n=U_n\varphi_n.
\end{equation}
Then $(e_n)_{n\geq0}$ and $(f_n)_{n\geq0}$ are orthonormal bases of
$\calM\ominus z\calM$ and $\calM\ominus w\calM$, respectively.  Moreover,
\[
 e_0=f_0=\frac p{\|p\|}.
\]
\end{theorem}

\begin{proof}
By Proposition~\ref{prop:slice},
\[
 pH_n\ominus zpH_{n-1}=U_n(\calP_n\ominus\zeta\calP_{n-1})
 =\C U_n\varphi_n^*,
\]
and
\[
 pH_n\ominus wpH_{n-1}=U_n(\calP_n\ominus\calP_{n-1})
 =\C U_n\varphi_n.
\]
Different $n$ lie in orthogonal homogeneous layers.  The conclusion follows
from \eqref{eq:z-wandering-decomp}--\eqref{eq:w-wandering-decomp}; the
formula at $n=0$ uses $\varphi_0=1$ because $\mu$ is a probability measure.
\end{proof}

\section{Projection products, the cross-commutator and the core}\label{sec:operators}

For vectors $u,v$ in a Hilbert space we write
$(u\otimes v)h=\langle h,v\rangle u$.  The homogeneous decomposition and
Theorem~\ref{thm:wandering-bases} give
\[
 P_z=\sum_{n\geq0}e_n\otimes e_n,
 \qquad
 P_w=\sum_{n\geq0}f_n\otimes f_n,
\]
with strong operator convergence.

\begin{theorem}[Schmidt decomposition]\label{thm:schmidt}
One has
\begin{equation}\label{eq:schmidt}
 P_zP_w=e_0\otimes f_0-
 \sum_{n\geq1}\overline{\alpha_{n-1}}\,e_n\otimes f_n,
\end{equation}
where the series converges strongly.  After absorbing phases into the
$e_n$, this is a Schmidt decomposition.  Thus the non-zero singular values
of $P_zP_w$, counted with multiplicity before decreasing rearrangement, are
\[
 1,|\alpha_0|,|\alpha_1|,\ldots .
\]
In particular, for every $r>0$,
\[
 P_zP_w\in\Sr_r\quad\Longleftrightarrow\quad(\alpha_n)\in\ell^r.
\]
\end{theorem}

\begin{proof}
The projection of $f_n$ onto $\calM\ominus z\calM$ belongs to the same
homogeneous layer and is
\[
 P_zf_n=\langle f_n,e_n\rangle e_n.
\]
For $n=0$ the coefficient is one.  For $n\geq1$,
Lemma~\ref{lem:overlaps} gives
\[
 \langle f_n,e_n\rangle
 =\langle\varphi_n,\varphi_n^*\rangle_\mu
 =-\overline{\alpha_{n-1}}.
\]
This proves \eqref{eq:schmidt}.  Its initial and final vector systems are
orthonormal, so the remaining assertions follow.
\end{proof}

\begin{theorem}[Cross-commutator]\label{thm:cross}
The cross-commutator $X=X_{\calM}$ satisfies
\begin{equation}\label{eq:cross-factor}
 X=R_z^*R_wP_z,
 \qquad
 Xe_n=\alpha_nf_n\quad(n\geq0),
\end{equation}
and it vanishes on $z\calM$.  Its non-zero singular values are therefore
$|\alpha_0|,|\alpha_1|,\ldots$, with multiplicity.
\end{theorem}

\begin{proof}
The coordinate isometries commute, and hence
\[
 R_z^*R_wP_z
 =R_z^*R_w(I-R_zR_z^*)
 =R_z^*R_w-R_wR_z^*=X.
\]
Moreover, $R_w^*X=0$, so $\Ran X\subset\calM\ominus w\calM$.  Homogeneity
forces $Xe_n$ to be a multiple of $f_n$, and
\[
 \langle Xe_n,f_n\rangle
 =\langle R_z^*R_we_n,f_n\rangle
 =\langle we_n,zf_n\rangle
 =\langle\varphi_n^*,\zeta\varphi_n\rangle_\mu
 =\alpha_n
\]
by \eqref{eq:overlap-shifted}.  This proves the result.
\end{proof}

The core operator has a particularly simple interpretation as a difference
of projections.  Since $R_z$ and $R_w$ commute,
\begin{equation}\label{eq:core-difference}
 C_{\calM}=P_z-R_wP_zR_w^*.
\end{equation}
The second term is the projection onto
$w(\calM\ominus z\calM)$.

\begin{lemma}[Difference of two rank-one projections]\label{lem:two-proj}
Let $u,v$ be unit vectors.  On $\Span\{u,v\}$ the non-zero eigenvalues of
$u\otimes u-v\otimes v$ are
\[
 \pm\sqrt{1-|\langle u,v\rangle|^2}.
\]
\end{lemma}

\begin{proof}
The operator has trace zero and the trace of its square is
$2(1-|\langle u,v\rangle|^2)$.  Its rank is at most two, which proves the
claim, including the degenerate case in which $u$ and $v$ are collinear.
\end{proof}

\begin{theorem}[Core spectrum]\label{thm:core-spectrum}
The non-zero spectrum of $C_{\calM}$, counted with multiplicity, is
\begin{equation}\label{eq:core-spectrum}
 \{1\}\cup\{\pm|\alpha_n|:n\geq0\},
\end{equation}
with zero completing the spectrum.  Hence, for every $r>0$,
\begin{equation}\label{eq:core-schatten}
 C_{\calM}\in\Sr_r\quad\Longleftrightarrow\quad(\alpha_n)\in\ell^r,
 \qquad
 \|C_{\calM}\|_{\Sr_r}^r=1+2\sum_{n\geq0}|\alpha_n|^r.
\end{equation}
\end{theorem}

\begin{proof}
The degree-$d$ block of \eqref{eq:core-difference} is the projection onto
$\C e_0$, and hence contributes the eigenvalue $1$.  For $n\geq1$, the
degree-$(d+n)$ block is the difference of the projections onto
\[
 \C e_n\quad\text{and}\quad\C R_we_{n-1}.
\]
Under $U_n$, the vector $R_we_{n-1}$ corresponds to
$\varphi_{n-1}^*\in\calP_n$.  The reversed Szeg\H{o} recursion gives
\[
 \varphi_{n-1}^*=\rho_{n-1}\varphi_n^*
        +\alpha_{n-1}\zeta\varphi_{n-1}.
\]
Since $\varphi_n^*\perp\zeta\calP_{n-1}$,
\[
 |\langle e_n,R_we_{n-1}\rangle|=\rho_{n-1}.
\]
Lemma~\ref{lem:two-proj} now gives the eigenvalues
$\pm(1-\rho_{n-1}^2)^{1/2}=\pm|\alpha_{n-1}|$.  Orthogonality of the
homogeneous blocks proves \eqref{eq:core-spectrum}, and
\eqref{eq:core-schatten} follows by summing the singular values.
\end{proof}

\begin{corollary}[Basic numerical invariants]\label{cor:sigma01}
Every homogeneous principal polynomial submodule is Hilbert--Schmidt, and
\begin{equation}\label{eq:sigma01}
 \Sigma_0([p])=1+\sum_{n\geq0}|\alpha_n|^2,
 \qquad
 \Sigma_1([p])=\sum_{n\geq0}|\alpha_n|^2.
\end{equation}
In particular, $\Sigma_0([p])-\Sigma_1([p])=1$ and
\[
 \|C_{[p]}\|_{\HS}^2=1+2\sum_{n\geq0}|\alpha_n|^2.
\]
\end{corollary}

\begin{proof}
Since $q$ is a non-zero polynomial, $\log|q|\in L^1(\T)$.  Thus the
probability measure \eqref{eq:mu-def} satisfies the Szeg\H{o} condition.
The Szeg\H{o}--Verblunsky theorem \cite[Theorem~2.2.5]{SimonOPUC1} gives
\[
 \prod_{n\geq0}(1-|\alpha_n|^2)>0.
\]
It follows from $x\leq-\log(1-x)$ that $(\alpha_n)\in\ell^2$.  The formulas
are now Theorems~\ref{thm:schmidt}, \ref{thm:cross}, and
\ref{thm:core-spectrum}.
\end{proof}

\begin{remark}\label{rem:trace-class}
Within the present polynomial class, Baxter's theorem
\cite[Chapter~5]{SimonOPUC1} shows that $(\alpha_n)\in\ell^1$ precisely when
$q$ has no zero on $\T$.  Therefore $C_{[p]}$ is trace class exactly when
$p(\zeta,1)$ has no unit-circle zero.  In the zero-free case the
Verblunsky coefficients in fact decay exponentially.
\end{remark}

\section{Toeplitz determinants and Mahler measure}\label{sec:toeplitz}

Let
\[
 \widehat\mu(k)=\int_{\T}\zeta^{-k}\,\dd\mu(\zeta),\qquad k\in\mathbb Z,
\]
and define
\begin{equation}\label{eq:Dn}
 D_0=1,
 \qquad
 D_n=\det\big(\widehat\mu(j-k)\big)_{j,k=0}^{n-1}\quad(n\geq1).
\end{equation}
The matrices are positive definite, so $D_n>0$.  Let $A_n$ denote the
$(n+1)\times(n+1)$ moment matrix, and let $C_n$ be the unsigned corner minor
obtained by deleting the first row and last column of $A_n$.

\begin{proposition}[Toeplitz endpoint identity]\label{prop:desnanot}
For $n\geq1$,
\begin{equation}\label{eq:desnanot}
 D_{n+1}D_{n-1}=D_n^2-|C_n|^2.
\end{equation}
Moreover,
\begin{equation}\label{eq:alpha-D}
 |\alpha_{n-1}|=\frac{|C_n|}{D_n},
 \qquad
 |\alpha_{n-1}|^2=1-\frac{D_{n-1}D_{n+1}}{D_n^2}.
\end{equation}
\end{proposition}

\begin{proof}
Apply the Desnanot--Jacobi identity to $A_n$, using its first and last rows
and columns.  The two principal $n\times n$ minors are $D_n$ by
Toeplitzness, the central $(n-1)\times(n-1)$ minor is $D_{n-1}$, and the
opposite corner minors are complex conjugates.  This proves
\eqref{eq:desnanot}.

Alternatively, the unit vectors spanning the two endpoint complements in
$\calP_n$ are $\varphi_n^*$ and $\varphi_n$.  Cramer's rule applied to the
dual basis of $1,\zeta,\ldots,\zeta^n$ gives
\[
 |\langle\varphi_n^*,\varphi_n\rangle|=\frac{|C_n|}{D_n}.
\]
The left-hand side is $|\alpha_{n-1}|$ by
Lemma~\ref{lem:overlaps}.  Combining this with \eqref{eq:desnanot} yields
\eqref{eq:alpha-D}.
\end{proof}

\begin{corollary}[Determinant forms of the singular data]\label{cor:det-forms}
The singular value associated with the $n$th positive homogeneous layer is
\[
 s_n(P_zP_w)=s_{n-1}(X_{\calM})
 =\left(1-\frac{D_{n-1}D_{n+1}}{D_n^2}\right)^{1/2},\qquad n\geq1.
\]
Consequently,
\[
 \|P_zP_w\|_{\HS}^2
 =1+\sum_{n\geq1}\left(1-\frac{D_{n-1}D_{n+1}}{D_n^2}\right).
\]
\end{corollary}

The exact product of the determinant-ratio defects is governed by Mahler
measure.  Set
\begin{equation}\label{eq:Mahler}
 \mathfrak M(q)=\exp\!\left(\int_{\T}\log|q(\zeta)|\,\dd m(\zeta)\right).
\end{equation}

\begin{theorem}[Szeg\H{o}--Mahler product]\label{thm:szego-mahler}
For the Verblunsky coefficients of \eqref{eq:mu-def},
\begin{equation}\label{eq:szego-mahler}
 \prod_{n=0}^{\infty}(1-|\alpha_n|^2)
 =\exp\!\left(\int_{\T}\log\frac{|q|^2}{\|q\|_2^2}\,\dd m\right)
 =\frac{\mathfrak M(q)^2}{\|q\|_2^2}>0.
\end{equation}
\end{theorem}

\begin{proof}
The function $\log|q|$ is integrable because the logarithmic singularities
at the finitely many unit-circle zeros of $q$ are integrable.  Apply the
Szeg\H{o}--Verblunsky product theorem to the probability measure
\eqref{eq:mu-def}; see \cite[Theorem~2.2.5]{SimonOPUC1} or
\cite[Chapter~5]{GrenanderSzego}.
\end{proof}

Let $Q=e_0\otimes e_0$, the projection onto the common line of the two
wandering spaces, and put $K=P_zP_w-Q$.

\begin{corollary}[Fredholm determinant and logarithmic trace]\label{cor:fredholm}
Both $K^*K$ and $X_{\calM}^*X_{\calM}$ are trace class, and
\begin{equation}\label{eq:fredholm}
 \det(I-K^*K)=\det(I-X_{\calM}^*X_{\calM})
 =\frac{\mathfrak M(q)^2}{\|q\|_2^2}.
\end{equation}
Furthermore,
\begin{equation}\label{eq:logtrace}
 \sum_{r=1}^{\infty}\frac1r\|X_{\calM}\|_{\Sr_{2r}}^{2r}
 =\log\frac{\|q\|_2^2}{\mathfrak M(q)^2}.
\end{equation}
The same identity holds with $X_{\calM}$ replaced by $K$.
\end{corollary}

\begin{proof}
Theorems~\ref{thm:schmidt} and \ref{thm:cross} show that $K$ and
$X_{\calM}$ have the same non-zero singular values $|\alpha_n|$.  Thus
\eqref{eq:fredholm} is \eqref{eq:szego-mahler}.  Finally, Tonelli's theorem
and $-\log(1-x)=\sum_{r\geq1}x^r/r$ give
\[
 -\log\prod_{n\geq0}(1-|\alpha_n|^2)
 =\sum_{r\geq1}\frac1r\sum_{n\geq0}|\alpha_n|^{2r},
\]
which is \eqref{eq:logtrace}.
\end{proof}

We finish this section with a sharp degree-dependent bound.  The elementary
coefficient estimate below is closely related to the Mahler-measure
inequality in \cite{StulovYang}.

\begin{lemma}[Coefficient norm versus Mahler measure]\label{lem:mahler-bound}
If $q$ is a non-zero polynomial of degree at most $d$, then
\begin{equation}\label{eq:mahler-bound}
 \|q\|_2^2\leq\binom{2d}{d}\mathfrak M(q)^2.
\end{equation}
If $\deg q=d$, equality holds precisely for
$q(\zeta)=c(\zeta-\beta)^d$ with $c\neq0$ and $|\beta|=1$.
\end{lemma}

\begin{proof}
Write $q(\zeta)=c\prod_{j=1}^d(\zeta-\beta_j)$, allowing zero roots if the
original degree is smaller than $d$.  Jensen's formula gives
\[
 \mathfrak M(q)=|c|\prod_{j=1}^d\max\{1,|\beta_j|\}.
\]
Every coefficient of $q$ is a sum of $\binom{d}{k}$ products of $k$ roots,
and each such product, after multiplication by $|c|$, has modulus at most
$\mathfrak M(q)$.  Hence
\[
 |\widehat q(k)|\leq\binom{d}{k}\mathfrak M(q)
 \quad(0\leq k\leq d).
\]
Parseval's identity and Vandermonde's identity yield
\[
 \|q\|_2^2\leq\mathfrak M(q)^2
 \sum_{k=0}^d\binom{d}{k}^2
 =\binom{2d}{d}\mathfrak M(q)^2.
\]
If equality holds and $\deg q=d$, equality in the leading and constant
coefficient bounds forces all roots to lie on $\T$.  Equality for the
coefficient of degree $d-1$ then forces the $d$ unit roots to have a common
argument.  The converse is immediate from
$\sum_k\binom dk^2=\binom{2d}d$.
\end{proof}

\begin{theorem}[A degree bound]\label{thm:degree-bound}
If $p$ is homogeneous and $d=\deg p(\cdot,1)$, then
\begin{equation}\label{eq:degree-final}
 \|X_{[p]}\|_{\HS}^2
 \leq\log\frac{\|p(\cdot,1)\|_2^2}{\mathfrak M(p(\cdot,1))^2}
 \leq\log\binom{2d}{d}<2d\log2.
\end{equation}
\end{theorem}

\begin{proof}
By $x\leq-\log(1-x)$ and Theorem~\ref{thm:szego-mahler},
\[
 \|X_{[p]}\|_{\HS}^2
 =\sum_{n\geq0}|\alpha_n|^2
 \leq-\sum_{n\geq0}\log(1-|\alpha_n|^2)
 =\log\frac{\|q\|_2^2}{\mathfrak M(q)^2}.
\]
Lemma~\ref{lem:mahler-bound} proves the second inequality, and
$\binom{2d}{d}<4^d$ for $d\geq1$ proves the last one.  The case $d=0$ is
trivial.
\end{proof}

\section{The circular Jacobi family and an unboundedness theorem}
\label{sec:unbounded}

Consider
\[
 p_N(z,w)=(z-w)^N,\qquad N\geq1.
\]
The corresponding unnormalized weight is the pure Fisher--Hartwig weight
$|1-\zeta|^{2N}\dd m(\zeta)$.  Its finite Toeplitz determinants are known
exactly \cite{BottcherSilbermannFH,BottcherWidom}:
\begin{equation}\label{eq:FH-det}
 \widetilde D_n=
 \frac{G(n+1)G(n+2N+1)G(N+1)^2}
 {G(n+N+1)^2G(2N+1)},\qquad n\geq0,
\end{equation}
where $G$ is the Barnes $G$-function.  Normalizing the weight only multiplies
$\widetilde D_n$ by a scalar to the power $n$ and therefore does not change
successive determinant ratios.  The corresponding circular Jacobi
Verblunsky coefficients are
\begin{equation}\label{eq:circular-jacobi-alpha}
 \alpha_n=-\frac{N}{n+N+1},\qquad n\geq0;
\end{equation}
see, for example, \cite[Example~8.2.5]{SimonOPUC1} and
\cite{JangSong}.

\begin{theorem}[Unbounded homogeneous core norms]\label{thm:unbounded}
Let $\calM_N=[(z-w)^N]$.  Then
\begin{align}
 \spec(C_{\calM_N})\setminus\{0\}
 &=\left\{1,\ \pm\frac{N}{N+1},\ \pm\frac{N}{N+2},\ldots\right\},
 \label{eq:spectrum-pN}\\
 \|X_{\calM_N}\|_{\HS}^2
 &=N^2\sum_{j=N+1}^{\infty}\frac1{j^2},\label{eq:cross-pN}\\
 \|C_{\calM_N}\|_{\HS}^2
 &=1+2N^2\sum_{j=N+1}^{\infty}\frac1{j^2}\label{eq:core-pN}\\
 &=2N+\frac1{3N}-\frac1{15N^3}+O(N^{-5}).\label{eq:core-asymptotic}
\end{align}
In particular,
\[
 \sup\{\|C_{[p]}\|_{\HS}:0\neq p\in\C[z,w]
       \text{ homogeneous}\}=\infty.
\]
Thus Problem~10 of \cite{YangSurvey} has a negative answer.
\end{theorem}

\begin{proof}
Equations \eqref{eq:spectrum-pN}--\eqref{eq:core-pN} follow immediately
from \eqref{eq:circular-jacobi-alpha}, Theorems~\ref{thm:cross} and
\ref{thm:core-spectrum}.  For completeness, the modulus in
\eqref{eq:circular-jacobi-alpha} also follows directly from
\eqref{eq:FH-det} and \eqref{eq:alpha-D}, because $G(x+1)=\Gamma(x)G(x)$
gives
\[
 \frac{\widetilde D_{n+1}\widetilde D_{n-1}}{\widetilde D_n^2}
 =\frac{n(n+2N)}{(n+N)^2}
 =1-\frac{N^2}{(n+N)^2},\qquad n\geq1.
\]
Finally, Euler--Maclaurin summation yields
\[
 \sum_{j=N+1}^{\infty}j^{-2}
 =\frac1N-\frac1{2N^2}+\frac1{6N^3}
  -\frac1{30N^5}+O(N^{-7}).
\]
Substitution into \eqref{eq:core-pN} proves
\eqref{eq:core-asymptotic} and the unboundedness assertion.
\end{proof}

\begin{remark}
The exact Fredholm determinant in this family is
\[
 \det(I-X_{\calM_N}^*X_{\calM_N})
 =\prod_{j=N+1}^{\infty}\left(1-\frac{N^2}{j^2}\right)
 =\binom{2N}{N}^{-1}.
\]
Indeed, $\mathfrak M((\zeta-1)^N)=1$ and
$\|(\zeta-1)^N\|_2^2=\binom{2N}{N}$, so this is also
Corollary~\ref{cor:fredholm}.  Lemma~\ref{lem:mahler-bound} is sharp on this
family, whereas the passage from the Fredholm determinant to the
Hilbert--Schmidt norm is necessarily strict for $N\geq1$.
\end{remark}

\section{Homogeneous quotient modules and compressed cross-commutators}
\label{sec:quotient}

We now pass from restrictions to compressions.  This distinction is
essential: the compressed coordinate shifts on a quotient need not be
isometries, and essential normality of homogeneous quotient modules is a
separate geometric question; see \cite{WangZhao2018,WangZhaoZhu,YangSurvey}.
Nevertheless, for the single cross-commutator considered here, the same OPUC
model gives a complete singular-value formula.

Let $p$ be homogeneous of total degree $D\geq1$, set
\[
 \mathcal Q_p=\HH\ominus[p],
 \qquad
 S_z=P_{\mathcal Q_p}M_z|_{\mathcal Q_p},\quad
 S_w=P_{\mathcal Q_p}M_w|_{\mathcal Q_p},
\]
and write
\begin{equation}\label{eq:q-endpoints}
 q(\zeta)=p(\zeta,1)=\sum_{j=0}^Dq_j\zeta^j,
 \qquad
 a_q=\frac{|q_0|^2}{\|q\|_2^2},\quad
 b_q=\frac{|q_D|^2}{\|q\|_2^2}.
\end{equation}
The extreme coefficients are allowed to vanish.  Let $\kappa_n>0$ be the
leading coefficient of $\varphi_n$.  Thus
\begin{equation}\label{eq:kappa-product}
 \kappa_0=1,\qquad
 \kappa_n=\prod_{j=0}^{n-1}\rho_j^{-1}\quad(n\geq1).
\end{equation}

The following block factorization explains why the quotient formula differs
from Theorem~\ref{thm:cross}.

\begin{lemma}[Restriction--compression factorization]
\label{lem:restriction-compression}
Relative to $\HH=[p]\oplus\mathcal Q_p$, write
\begin{equation}\label{eq:ambient-block-shifts}
 M_z=\begin{pmatrix}R_z&D_z\\0&S_z\end{pmatrix},
 \qquad
 M_w=\begin{pmatrix}R_w&D_w\\0&S_w\end{pmatrix}.
\end{equation}
Then
\begin{equation}\label{eq:ABBA}
 [R_z^*,R_w]=D_wD_z^*,
 \qquad
 [S_z^*,S_w]=-D_z^*D_w.
\end{equation}
Moreover,
\[
 \Ran D_z\subset[p]\ominus z[p],
 \qquad
 \Ran D_w\subset[p]\ominus w[p].
\]
\end{lemma}

\begin{proof}
The ambient coordinate shifts doubly commute, so
$M_z^*M_w=M_wM_z^*$.  Comparing the upper-left and lower-right blocks in
\eqref{eq:ambient-block-shifts} gives \eqref{eq:ABBA}.  Comparing the
upper-right block in $M_z^*M_z=I$ gives $R_z^*D_z=0$, and the corresponding
identity for $w$ gives $R_w^*D_w=0$.  These are precisely the two range
inclusions.
\end{proof}

\begin{remark}[The $AB$ versus $BA$ principle]
The restricted and compressed cross-commutators are the two products $AB$
and $-BA$.  Their non-zero eigenvalues are related, but their singular values
need not agree.  In the homogeneous setting every degree block of $D_z$ and
$D_w$ has one-dimensional range.  This rank-one geometry is what makes the
singular values of $-BA$ explicitly computable.
\end{remark}

For $n\geq0$ let
\[
 \mathcal Q_{D+n-1}=\mathcal Q_p\cap H_{D+n-1},
\]
where $\mathcal Q_{D-1}=H_{D-1}$.  For $n\geq1$ one has
\[
 \mathcal Q_{D+n-1}=H_{D+n-1}\ominus pH_{n-1},
 \qquad \dim\mathcal Q_{D+n-1}=D.
\]
Put
\begin{equation}\label{eq:quotient-uv}
 u_n=D_z^*e_n=M_z^*e_n,
 \qquad
 v_n=D_w^*f_n=M_w^*f_n.
\end{equation}
The last equalities hold because $R_z^*e_n=R_w^*f_n=0$.

\begin{lemma}[Boundary defect norms]\label{lem:boundary-defects}
For every $n\geq0$,
\begin{equation}\label{eq:uv-norms}
 \|u_n\|^2=1-a_q\kappa_n^2,
 \qquad
 \|v_n\|^2=1-b_q\kappa_n^2.
\end{equation}
\end{lemma}

\begin{proof}
Identify $H_{D+n}$ with $\mathcal P_{D+n}\subset H^2(\T)$ by the unitary
homogenization
\[
 J_{D+n}F(z,w)=w^{D+n}F(z/w).
\]
By Theorem~\ref{thm:wandering-bases},
\[
 e_n=J_{D+n}\!\left(\frac{q\varphi_n^*}{\|q\|_2}\right),
 \qquad
 f_n=J_{D+n}\!\left(\frac{q\varphi_n}{\|q\|_2}\right).
\]
On $H_{D+n}$, the kernel of $M_z^*$ is the line spanned by $w^{D+n}$.
The coefficient of this monomial in $e_n$ is
$q_0\varphi_n^*(0)/\|q\|_2=q_0\kappa_n/\|q\|_2$.
Since $e_n$ is a unit vector, deleting that coefficient gives
\[
 \|M_z^*e_n\|^2
 =1-\frac{|q_0|^2\kappa_n^2}{\|q\|_2^2}.
\]
Similarly, the kernel of $M_w^*$ in $H_{D+n}$ is spanned by $z^{D+n}$.
The coefficient of $z^{D+n}$ in $f_n$ is
$q_D\kappa_n/\|q\|_2$, which proves the second identity.
\end{proof}

\begin{theorem}[Singular values on the homogeneous quotient]
\label{thm:quotient-singular-values}
Let
\[
 Y_p=[S_z^*,S_w]\quad\text{on }\mathcal Q_p.
\]
The operator $Y_p$ vanishes on
$\bigoplus_{m=0}^{D-2}H_m$ (with the usual empty-sum convention).  The
subspaces $\mathcal Q_{D+n-1}$ reduce $Y_p$, and the restriction of $Y_p$ to
each of them has rank at most one.  Its only possible non-zero singular value
is
\begin{align}
 \beta_0(p)&=\sqrt{(1-a_q)(1-b_q)},\label{eq:beta0}\\
 \beta_n(p)&=|\alpha_{n-1}|
 \sqrt{(1-a_q\kappa_n^2)(1-b_q\kappa_n^2)},
 \qquad n\geq1.\label{eq:betan}
\end{align}
Consequently, after zero terms are removed,
\[
 s(Y_p)=\{\beta_0(p),\beta_1(p),\beta_2(p),\ldots\}
\]
as a multiset before decreasing rearrangement.
\end{theorem}

\begin{proof}
Both products in \eqref{eq:ABBA} preserve total degree.  By
Lemma~\ref{lem:restriction-compression}, the degree-$(D+n)$ ranges of $D_z$
and $D_w$ are contained in the lines $\C e_n$ and $\C f_n$, respectively.
If $m<D-1$, multiplication carries $H_m$ into $H_{m+1}\perp[p]$, so both
$D_z$ and $D_w$ vanish on $H_m$ and hence so does $Y_p$.
Hence, on $\mathcal Q_{D+n-1}$,
\[
 D_z=e_n\otimes u_n,\qquad D_w=f_n\otimes v_n.
\]
It follows that
\begin{equation}\label{eq:quotient-rank-one-block}
 Y_p|_{\mathcal Q_{D+n-1}}
 =-\langle f_n,e_n\rangle\,u_n\otimes v_n.
\end{equation}
At $n=0$ one has $e_0=f_0$.  For $n\geq1$,
Lemma~\ref{lem:overlaps} gives
\[
 |\langle f_n,e_n\rangle|=|\alpha_{n-1}|.
\]
The singular value of $u\otimes v$ is $\|u\|\|v\|$.
Lemma~\ref{lem:boundary-defects} now proves
\eqref{eq:beta0}--\eqref{eq:betan}.
\end{proof}

\begin{remark}[What the three factors measure]
The overlap $|\alpha_{n-1}|$ measures the angle between the two wandering
lines in degree $D+n$.  The factors
$\sqrt{1-a_q\kappa_n^2}$ and $\sqrt{1-b_q\kappa_n^2}$ measure their coupling
to the preceding quotient block through $M_z^*$ and $M_w^*$, respectively.
Thus the quotient remembers the choice of polynomial spectral factor through
the two extreme coefficients of $q$; the Verblunsky sequence alone need not
determine its compressed cross-commutator.
\end{remark}

\begin{corollary}[Hilbert--Schmidt formula and degree estimate]
\label{cor:quotient-HS}
The quotient cross-commutator $Y_p$ is Hilbert--Schmidt and
\begin{equation}\label{eq:quotient-HS-exact}
 \begin{split}
 \|Y_p\|_{\HS}^2={}&(1-a_q)(1-b_q)\\
 &+\sum_{n=1}^{\infty}|\alpha_{n-1}|^2
 (1-a_q\kappa_n^2)(1-b_q\kappa_n^2).
 \end{split}
\end{equation}
Moreover,
\begin{equation}\label{eq:quotient-HS-bound}
 \|Y_p\|_{\HS}^2
 \leq1+\sum_{n\geq0}|\alpha_n|^2
 \leq1+\log\frac{\|q\|_2^2}{\mathfrak M(q)^2}
 \leq1+\log\binom{2D}{D}<1+2D\log2.
\end{equation}
\end{corollary}

\begin{proof}
Equation \eqref{eq:quotient-HS-exact} follows by summing the mutually
orthogonal rank-one blocks in Theorem~\ref{thm:quotient-singular-values}.
The two factors in every summand lie in $[0,1]$ by
Lemma~\ref{lem:boundary-defects}.  The remaining inequalities are the proof
of Theorem~\ref{thm:degree-bound}, with $\deg q\leq D$.
\end{proof}

\begin{example}[Monomial generators]\label{ex:quotient-monomial}
If $p=z^rw^{D-r}$, then $\mu$ is Haar measure and every Verblunsky
coefficient vanishes.  If $r=0$ or $r=D$, one of $a_q,b_q$ equals one, and
Theorem~\ref{thm:quotient-singular-values} gives
$[S_z^*,S_w]=0$.  By contrast, if $0<r<D$, then $a_q=b_q=0$; the degree
$D-1$ block has singular value $\beta_0=1$, and all later blocks vanish.
Thus the compressed cross-commutator has rank one and norm one.  This sharp
dichotomy is already visible for $p=zw$: on the span of $z,w$ in the
degree-one quotient block, the commutator sends $z$ to $-w$ and annihilates
the orthogonal line.

Even in the pure-power case, cross-commutation does not assert that the two
compressed shifts are essentially normal: the two self-commutators are
separate conditions.  For example, $[z^2]^\perp$ is not essentially normal;
compare \cite{WangZhaoZhu}.
\end{example}

\begin{example}[The diagonal and the Bergman shift]\label{ex:z-w-quotient}
For $p=z-w$, the quotient $[z-w]^\perp$ identifies both compressed shifts
with the Bergman shift; see \cite{GuoSunZhengZhong,YangSurvey}.  Here
\[
 \alpha_n=-\frac1{n+2},\qquad
 \kappa_n^2=\frac{2(n+1)}{n+2},\qquad a_q=b_q=\frac12.
\]
Theorem~\ref{thm:quotient-singular-values} gives
\[
 \beta_0=\frac12,\qquad
 \beta_n=\frac1{(n+1)(n+2)}\quad(n\geq1).
\]
These are exactly the diagonal entries of the Bergman self-commutator, and
\begin{equation}\label{eq:Bergman-HS-value}
 \|[S_z^*,S_w]\|_{\HS}^2
 =\frac14+\sum_{n=1}^{\infty}\frac1{(n+1)^2(n+2)^2}
 =\frac{\pi^2}{3}-3.
\end{equation}
\end{example}

The next family proves that the degree dependence in
\eqref{eq:quotient-HS-bound} cannot be removed.

\begin{theorem}[No uniform quotient bound]\label{thm:quotient-unbounded}
Let $p_N=(z-w)^N$ and let $Y_N$ denote the compressed cross-commutator on
$\mathcal Q_{p_N}$.  Put
\begin{equation}\label{eq:t-nN}
 t_{n,N}=\frac{(n+N)!^2}{n!(n+2N)!}.
\end{equation}
Then
\begin{align}
 \beta_{0,N}&=1-\binom{2N}{N}^{-1},\label{eq:beta0N}\\
 \beta_{n,N}&=\frac{N}{n+N}(1-t_{n,N}),\qquad n\geq1,
 \label{eq:betanN}
\end{align}
are all the non-zero singular values of $Y_N$.  Furthermore,
\begin{equation}\label{eq:quotient-linear-bounds}
 \frac N4\left(1-\left(\frac23\right)^N\right)^2
 \leq\|Y_N\|_{\HS}^2<N+1,
\end{equation}
and, more precisely,
\begin{equation}\label{eq:quotient-HS-asymptotic}
 \lim_{N\to\infty}\frac{\|Y_N\|_{\HS}^2}{N}=1.
\end{equation}
In particular, homogeneous quotient cross-commutators have no uniform
Hilbert--Schmidt bound.
\end{theorem}

\begin{proof}
Equation \eqref{eq:circular-jacobi-alpha} gives
$|\alpha_{n-1}|=N/(n+N)$ for $n\geq1$.  From
\eqref{eq:kappa-product},
\[
 \kappa_n^2
 =\prod_{j=0}^{n-1}\frac{(j+N+1)^2}{(j+1)(j+2N+1)}
 =\binom{2N}{N}\frac{(n+N)!^2}{n!(n+2N)!}.
\]
Since $|q_0|=|q_N|=1$ and
$\|q\|_2^2=\binom{2N}{N}$, Theorem
\ref{thm:quotient-singular-values} proves
\eqref{eq:beta0N}--\eqref{eq:betanN}.

For $1\leq n\leq N$,
\[
 t_{n,N}=\prod_{k=1}^N\frac{n+k}{n+N+k}
 \leq\left(\frac23\right)^N,
 \qquad \frac{N}{n+N}\geq\frac12.
\]
Summing the first $N$ non-constant blocks gives the lower bound in
\eqref{eq:quotient-linear-bounds}.  Conversely,
\[
 \beta_{n,N}\leq\frac{N}{n+N},
\]
so the integral test yields
\[
 \|Y_N\|_{\HS}^2
 \leq1+N^2\sum_{n=1}^{\infty}\frac1{(n+N)^2}<N+1.
\]

It remains to identify the sharp linear coefficient.  The last upper bound,
divided by $N$, has limit at most one.  Fix $L>0$.  Stirling's formula gives,
uniformly for $1\leq n\leq LN$,
\[
 t_{n,N}\longrightarrow0
\]
exponentially away from the harmless endpoint $n=0$; equivalently, for
$n/N\to x\in[0,L]$,
\[
 \frac1N\log t_{n,N}\longrightarrow
 2(x+1)\log(x+1)-x\log x-(x+2)\log(x+2)<0.
\]
Therefore the Riemann lower sum over $1\leq n\leq LN$ gives
\[
 \liminf_{N\to\infty}\frac{\|Y_N\|_{\HS}^2}{N}
 \geq\int_0^L\frac{\dd x}{(1+x)^2}=\frac{L}{L+1}.
\]
Letting $L\to\infty$ proves \eqref{eq:quotient-HS-asymptotic}.
\end{proof}

\section{Higher invariants and alternating CMV products}\label{sec:cmv}

For the canonical wandering bases in Theorem~\ref{thm:wandering-bases},
define, for $k\geq0$,
\begin{equation}\label{eq:sigma-k-def}
 \Sigma_k([p])=
 \sum_{m,n\geq0}|\langle w^ke_m,z^kf_n\rangle|^2.
\end{equation}
This agrees with $\Sigma_0$ and $\Sigma_1$ above and with the higher
invariants considered in \cite{LiuLuZu}.  Equality of total degrees forces
$m=n$, so if
\begin{equation}\label{eq:Fkn}
 F_k(n)=\langle\varphi_n^*,\zeta^k\varphi_n\rangle_\mu,
\end{equation}
then
\begin{equation}\label{eq:sigma-F}
 \Sigma_k([p])=\sum_{n\geq0}|F_k(n)|^2.
\end{equation}
In particular,
\[
 F_0(0)=1,\quad F_0(n)=-\alpha_{n-1}\ (n\geq1),
 \qquad F_1(n)=\alpha_n.
\]

We recall the CMV factorization in the convention appropriate to inner
products linear in the first variable.  Put
\[
 \Theta(\alpha)=
 \begin{pmatrix}
   \alpha&\rho\\
   \rho&-\overline\alpha
 \end{pmatrix},
 \qquad \rho=(1-|\alpha|^2)^{1/2},
\]
and define
\begin{align}
 \mathbf L&=\Theta(\alpha_0)\oplus\Theta(\alpha_2)
              \oplus\Theta(\alpha_4)\oplus\cdots,
 \label{eq:L-def}\\
 \mathbf M&=1\oplus\Theta(\alpha_1)\oplus\Theta(\alpha_3)
              \oplus\cdots,
 \qquad \mathbf C=\mathbf L\mathbf M.\label{eq:M-C-def}
\end{align}
The harmless transpose-conjugate change from some standard texts is caused
only by the convention on which inner-product variable is linear.

For clarity, the CMV and alternate CMV Laurent bases are
\begin{align*}
 \chi_{2m}&=\zeta^{-m}\varphi_{2m}^*,
 &\chi_{2m-1}&=\zeta^{-m+1}\varphi_{2m-1},\\
 x_{2m}&=\zeta^{-m}\varphi_{2m},
 &x_{2m-1}&=\zeta^{-m}\varphi_{2m-1}^*.
\end{align*}
The Szeg\H{o} recursions give the two block factors
\eqref{eq:L-def}--\eqref{eq:M-C-def}, and $\mathbf C$ is the matrix of
multiplication by $\zeta$ in the $\chi$-basis; see
\cite{CanteroMoralVelazquez,SimonCMV,SimonOPUC2}.

\begin{theorem}[CMV formula for the higher invariants]\label{thm:cmv-formula}
For every $k\geq1$,
\begin{align}
 F_k(2m)&=(\mathbf C^{k-1}\mathbf L)_{2m,2m},\label{eq:F-even}\\
 F_k(2m+1)&=(\mathbf M\mathbf C^{k-1})_{2m+1,2m+1}.
 \label{eq:F-odd}
\end{align}
Consequently,
\begin{equation}\label{eq:sigma-cmv}
 \Sigma_k([p])=
 \sum_{m\geq0}|(\mathbf C^{k-1}\mathbf L)_{2m,2m}|^2
 +\sum_{m\geq0}|(\mathbf M\mathbf C^{k-1})_{2m+1,2m+1}|^2.
\end{equation}
For fixed $k$, $F_k(n)$ depends only on
$\alpha_{n-k+1},\ldots,\alpha_{n+k-1}$, apart from the finite left boundary.
Moreover, $(F_k(n))_{n\geq0}\in\ell^2$.
\end{theorem}

\begin{proof}
If $n=2m$, then $\varphi_n^*=\zeta^m\chi_n$ and
$\varphi_n=\zeta^mx_n$, whence
\[
 F_k(n)=\langle\chi_n,\zeta^kx_n\rangle.
\]
The change from the alternate CMV basis to the CMV basis after one
multiplication is represented by $\mathbf L$, while every subsequent
multiplication by $\zeta$ is represented by $\mathbf C$.  This proves
\eqref{eq:F-even}.  If $n=2m+1$, then
$\varphi_n=\zeta^m\chi_n$ and
$\varphi_n^*=\zeta^{m+1}x_n$, so
\[
 F_k(n)=\langle x_n,\zeta^{k-1}\chi_n\rangle,
\]
which gives \eqref{eq:F-odd}.  Formula \eqref{eq:sigma-cmv} follows from
\eqref{eq:sigma-F}.

Both $\mathbf L$ and $\mathbf M$ are direct sums of nearest-neighbour
$2\times2$ blocks, proving locality.  Expanding either diagonal entry as a
finite path sum shows that, outside the left boundary, every monomial
contains at least one of the finitely many local Verblunsky coefficients.
Since $|\alpha_j|,\rho_j\leq1$, a finite sum of shifted copies of the
$\ell^2$ sequence $(\alpha_j)$ dominates $|F_k(n)|$ up to a constant
depending only on $k$.  Corollary~\ref{cor:sigma01} therefore gives
$F_k\in\ell^2$.
\end{proof}

\begin{remark}[Why unitarity does not give monotonicity]\label{rem:no-monotonicity}
Although $\mathbf L$, $\mathbf M$ and $\mathbf C$ are unitary,
\eqref{eq:sigma-cmv} retains only selected parity-diagonal entries of
different unitary words.  There is no general inequality
$\|\operatorname{diag}(UA)\|_2\leq\|\operatorname{diag}(A)\|_2$ for a
unitary $U$: off-diagonal mass may return to the diagonal after another
unitary step.  The next section shows that this is not merely a limitation of
the argument.
\end{remark}

When all Verblunsky coefficients are real, direct multiplication of the
blocks gives, with $\alpha_{-1}=-1$ and $\rho_{-1}=0$,
\begin{align}
 F_1(n)&=\alpha_n,\label{eq:F1-real}\\
 F_2(n)&=\rho_n^2\alpha_{n+1}-\alpha_{n-1}\alpha_n^2,
 \label{eq:F2-real}\\
 F_3(n)&=\rho_n^2\rho_{n+1}^2\alpha_{n+2}
 -\rho_n^2\alpha_n\alpha_{n+1}^2
 -2\rho_n^2\alpha_{n-1}\alpha_n\alpha_{n+1}\notag\\
 &\hspace{2.8em}+\alpha_{n-1}^2\alpha_n^3
 -\rho_{n-1}^2\alpha_{n-2}\alpha_n^2.\label{eq:F3-real}
\end{align}
For general $k$, $F_k(n)$ is a finite sum over closed alternating CMV paths.
The unique path reaching the rightmost possible site contributes
\[
 \left(\prod_{j=n}^{n+k-2}\rho_j^2\right)\alpha_{n+k-1},
\]
while the remaining paths are nonlinear correction terms.

\section{A quadratic counterexample to monotonicity}\label{sec:counterexample}

Let
\begin{equation}\label{eq:pstar}
 p_*(z,w)=z^2-\sqrt3zw+w^2,
 \qquad q_*(\zeta)=\zeta^2-\sqrt3\zeta+1.
\end{equation}
Since $\|q_*\|_2^2=5$, the associated probability measure is
\begin{equation}\label{eq:mustar}
 \dd\mu_*(\zeta)=\frac{|\zeta^2-\sqrt3\zeta+1|^2}{5}\,\dd m(\zeta).
\end{equation}
Its only non-zero positive moments are
\[
 \widehat\mu_*(1)=-\frac{2\sqrt3}{5},
 \qquad
 \widehat\mu_*(2)=\frac15.
\]
Thus its Carath\'eodory and Schur functions are
\begin{align}
 \mathcal F(\zeta)&=1-\frac{4\sqrt3}{5}\zeta+\frac25\zeta^2,
 \label{eq:Caratheodory}\\
 f_0(\zeta)&=\frac{\mathcal F(\zeta)-1}
 {\zeta(\mathcal F(\zeta)+1)}
 =\frac{\zeta-2\sqrt3}{\zeta^2-2\sqrt3\zeta+5}.
 \label{eq:Schur0}
\end{align}
For real Schur data the Schur algorithm is
\begin{equation}\label{eq:Schur-algorithm}
 \alpha_n=f_n(0),
 \qquad
 f_{n+1}(\zeta)=\frac{f_n(\zeta)-\alpha_n}
 {\zeta(1-\alpha_nf_n(\zeta))}.
\end{equation}

\begin{lemma}[Six-step Verblunsky formula]\label{lem:six-step}
The Verblunsky coefficients of $\mu_*$ are real.  For $m\geq0$ they are
\begin{align}
 \alpha_{6m}&=(-1)^{m+1}\frac{2\sqrt3}{36m(m+1)+5},
 \label{eq:a0}\\
 \alpha_{6m+1}&=(-1)^{m+1}\frac{6m+7}{36m^2+48m+13},
 \label{eq:a1}\\
 \alpha_{6m+2}&=(-1)^{m+1}\frac{\sqrt3}{6m+4},
 \label{eq:a2}\\
 \alpha_{6m+3}&=(-1)^{m+1}\frac1{3m+3},
 \label{eq:a3}\\
 \alpha_{6m+4}&=(-1)^{m+1}\frac{\sqrt3}{6m+8},
 \label{eq:a4}\\
 \alpha_{6m+5}&=(-1)^{m+1}\frac{6m+5}{36m^2+96m+61}.
 \label{eq:a5}
\end{align}
\end{lemma}

\begin{proof}
Starting from \eqref{eq:Schur0}, apply \eqref{eq:Schur-algorithm}.
The first six values are
\[
 -\frac{2\sqrt3}{5},\quad -\frac7{13},\quad
 -\frac{\sqrt3}{4},\quad-\frac13,\quad
 -\frac{\sqrt3}{8},\quad-\frac5{61}.
\]
For a fully checkable induction, Appendix~\ref{app:schur-certificate} lists
the four coefficients of every rational function
\[
 f_{6m+r}(\zeta)=
 \frac{A_r(m)+B_r(m)\zeta}{1+C_r(m)\zeta+D_r(m)\zeta^2},
 \qquad 0\leq r\leq5.
\]
Substitution in \eqref{eq:Schur-algorithm} takes row $r$ to row $r+1$ and
takes row $5$ to row $0$ with $m$ replaced by $m+1$.  The constant
coefficients $A_r(m)$ are exactly \eqref{eq:a0}--\eqref{eq:a5}.  This proves
the formulas by induction.  The displayed moduli are all strictly less than
one, as required.
\end{proof}

Write $F_j(6m+r)=F_{j,r}(m)$ and define the six-term difference
\begin{equation}\label{eq:Bm-def}
 B_m=\sum_{r=0}^5\bigl(|F_4(6m+r)|^2-|F_3(6m+r)|^2\bigr).
\end{equation}
The exact alternating-CMV multiplication is given in
Appendix~\ref{app:cmv-certificate}.  It implies, for $m\geq1$,
\begin{equation}\label{eq:Bm-PQ}
 B_m=-18\frac{P(m)}{Q(m)},
\end{equation}
where $P$ and $Q$ are the integer polynomials in
\eqref{eq:Q-poly} and \eqref{eq:P-poly}.

\begin{lemma}[Block tail estimate]\label{lem:block-tail}
For every integer $m\geq1$,
\begin{equation}\label{eq:block-tail}
 -\frac1{6m^3}<B_m<0.
\end{equation}
\end{lemma}

\begin{proof}
The denominator $Q(m)$ is positive.  Although the last coefficients of
$P(m)$ in \eqref{eq:P-poly} are negative, direct expansion shows that
$P(m+1)$ has strictly positive coefficients.  Hence $P(m)>0$ for every
integer $m\geq1$, proving $B_m<0$.

For the lower estimate, clear denominators in
$B_m+1/(6m^3)$.  Its denominator is $6m^3Q(m)>0$, and its numerator is the
polynomial $R(m)$ in \eqref{eq:R-poly}.  Every coefficient of $R$ is
strictly positive, so $B_m+1/(6m^3)>0$.
\end{proof}

The initial CMV block has to be computed separately because of the singleton
block in $\mathbf M$.  Appendix~\ref{app:cmv-certificate} gives
\begin{equation}\label{eq:B0}
 B_0=\frac{14543757}{503079200}>\frac{289}{10000}.
\end{equation}
Substitution into \eqref{eq:Bm-PQ} also gives the convenient exact bounds
\begin{equation}\label{eq:B1-B4}
 B_1>-\frac{153}{10000},\quad
 B_2>-\frac{523}{100000},\quad
 B_3>-\frac{231}{100000},\quad
 B_4>-\frac{122}{100000}.
\end{equation}
Consequently,
\begin{equation}\label{eq:first-five}
 \sum_{m=0}^4B_m>\frac{121}{25000}.
\end{equation}

\begin{theorem}[Failure of monotonicity]\label{thm:counterexample}
For $\calM_*=[z^2-\sqrt3zw+w^2]$,
\begin{equation}\label{eq:counterexample-final}
 \Sigma_4(\calM_*)-\Sigma_3(\calM_*)>\frac{13}{75000}>0.
\end{equation}
Thus the sequence of higher numerical invariants need not be decreasing,
even for a singly generated quadratic homogeneous submodule.
\end{theorem}

\begin{proof}
By \eqref{eq:sigma-F} and \eqref{eq:Bm-def},
\[
 \Sigma_4(\calM_*)-\Sigma_3(\calM_*)=\sum_{m=0}^{\infty}B_m.
\]
Lemma~\ref{lem:block-tail} and the integral test give
\[
 \sum_{m=5}^{\infty}B_m
 >-\frac16\sum_{m=5}^{\infty}\frac1{m^3}
 \geq-\frac16\left(\frac1{5^3}+\int_5^{\infty}x^{-3}\,\dd x\right)
 =-\frac7{1500}.
\]
Combining this with \eqref{eq:first-five} yields
\[
 \Sigma_4(\calM_*)-\Sigma_3(\calM_*)
 >\frac{121}{25000}-\frac7{1500}
 =\frac{13}{75000}.
\]
No numerical approximation enters the argument.
\end{proof}

\begin{remark}
For orientation only,
\[
 \Sigma_3(\calM_*)\approx0.262378,
 \qquad \Sigma_4(\calM_*)\approx0.264677.
\]
The earlier monotonicity calculations for $[z-w]$ and $[(z-w)^2]$
therefore reflect special Verblunsky data rather than a general contraction
principle; compare \cite{LiuLuZu}.
\end{remark}

\section{Quasi-homogeneous generators}\label{sec:quasi}

We record the weighted extension in a form which makes clear that it contains
no new non-zero singular data.  Fix relatively prime $s,t\in\N$, define
$\deg_{s,t}(z^iw^j)=si+tj$, and suppose that
\begin{equation}\label{eq:quasi-p}
 p(z,w)=\sum_{si+tj=N}c_{ij}z^iw^j\neq0.
\end{equation}

\begin{lemma}[Homogeneous reduction]\label{lem:quasi-reduction}
There are unique integers $0\leq a<t$, $0\leq b<s$ and $d\geq0$ such that
\[
 N=sa+tb+std
\]
and
\begin{equation}\label{eq:quasi-reduction}
 p(z,w)=z^aw^bq(Z,W),
 \qquad Z=z^t,\quad W=w^s,
\end{equation}
where $q$ is an ordinary homogeneous polynomial of degree $d$.
\end{lemma}

\begin{proof}
If $(i,j)$ and $(i',j')$ solve $si+tj=N$, then
$s(i-i')=-t(j-j')$.  Coprimality implies $t\mid(i-i')$ and
$s\mid(j-j')$.  Thus all admissible $i$ and $j$ have fixed residues
$a\pmod t$ and $b\pmod s$, respectively.  Writing
$i=a+tr$ and $j=b+s\ell$ gives $r+\ell=d$ and proves the assertion and its
uniqueness.
\end{proof}

Put $Q(\zeta)=q(\zeta,1)$ and let $(\varphi_m)$ and $(\alpha_m)$ be the
OPUC and Verblunsky data of the normalized measure
$|Q|^2\dd m/\|Q\|_2^2$.  The character $(z,w)\mapsto z^tw^{-s}$ pushes Haar
measure on $\T^2$ forward to Haar measure on $\T$.

\begin{theorem}[Quasi-homogeneous wandering bases]\label{thm:quasi-bases}
Let $\calM=[p]$.  For $m\geq0$, $0\leq\rho<s$ and $0\leq\sigma<t$, set
\begin{align}
 e_{m,\rho}(z,w)
 &=\frac{p(z,w)}{\|p\|}\,w^\rho W^m
     \varphi_m^*(Z/W),\label{eq:quasi-e}\\
 f_{m,\sigma}(z,w)
 &=\frac{p(z,w)}{\|p\|}\,z^\sigma W^m
     \varphi_m(Z/W).\label{eq:quasi-f}
\end{align}
Then $(e_{m,\rho})$ is an orthonormal basis of $\calM\ominus z\calM$ and
$(f_{m,\sigma})$ is an orthonormal basis of $\calM\ominus w\calM$.  Moreover,
\begin{equation}\label{eq:quasi-mixed}
 \langle e_{m,\rho},f_{k,\sigma}\rangle=0
 \quad\text{unless}\quad m=k\ \text{and}\ \rho=\sigma=0.
\end{equation}
In the exceptional case the inner product is
$1$ for $m=0$ and $-\alpha_{m-1}$ for $m\geq1$.
\end{theorem}

\begin{proof}
The weighted homogeneous layer contributing to $\calM\ominus z\calM$ can
be indexed uniquely by $sm+\rho$, with $0\leq\rho<s$, and has the form
\[
 w^\rho\Span\{Z^rW^{m-r}:0\leq r\leq m\}.
\]
The corresponding part of $z\calM$ omits the term $r=0$.  The map
\[
 h\longmapsto\frac p{\|p\|}w^\rho W^mh(Z/W)
\]
is an isometry from $\calP_m\subset L^2(|Q|^2\dd m/\|Q\|_2^2)$ onto the
slice, and it maps $\zeta\calP_{m-1}$ onto the $z\calM$ part.  Its endpoint
orthogonal complement is therefore generated by $\varphi_m^*$, proving
\eqref{eq:quasi-e}.  The same argument, writing the relevant index as
$tm+\sigma$, gives \eqref{eq:quasi-f}.

The weighted degrees of $e_{m,\rho}$ and $f_{k,\sigma}$ can agree only if
\[
 st(m-k)=s\sigma-t\rho.
\]
The right-hand side has absolute value strictly smaller than $st$.
Coprimality and the prescribed residue ranges therefore force
$m=k$ and $\rho=\sigma=0$.  In that case the inner product reduces to
$\langle\varphi_m^*,\varphi_m\rangle$, and
Lemma~\ref{lem:overlaps} completes the proof.
\end{proof}

\begin{corollary}[No new non-zero singular values]\label{cor:quasi-singular}
For an $(s,t)$-quasi-homogeneous generator, the non-zero singular values of
$P_zP_w$ are
\[
 1,|\alpha_0|,|\alpha_1|,\ldots .
\]
They are exactly the non-zero singular values associated with the ordinary
homogeneous submodule $[q]\subset H^2(\D^2)$.  The weights $s,t$ and the
monomial factor $z^aw^b$ contribute only mutually orthogonal zero blocks.
\end{corollary}

\begin{proof}
Expand the two wandering projections in the bases of
Theorem~\ref{thm:quasi-bases}.  Equation \eqref{eq:quasi-mixed} annihilates
every product block except the common zero-residue block, where the
coefficients are $1,-\overline{\alpha_0},-\overline{\alpha_1},\ldots$.
\end{proof}

For example, if $p(z,w)=(z^t-w^s)^d$, then the ordinary reduction is
$q(Z,W)=(Z-W)^d$.  Hence
\begin{equation}\label{eq:quasi-model-singular}
 s_0(P_zP_w)=1,
 \qquad
 s_m(P_zP_w)=\frac d{m+d}\quad(m\geq1),
\end{equation}
independently of the relatively prime weights $s,t$.

\section{Beyond the graded setting: general polynomial generators}
\label{sec:general}

We now isolate the part of the construction that remains valid without
homogeneity.  The resulting model is exact, but it is two-variable: in
general there is no scalar measure on $\T$ and no single Verblunsky sequence.
This distinction is consistent with the recurrence and moment-matrix
phenomena in the theory of orthogonal polynomials on the bicircle; see
\cite{GeronimoWoerdemanAnnals,GeronimoWoerdemanBicircle}.

Let $0\neq p\in\C[z,w]$ and write its total-degree decomposition as
\begin{equation}\label{eq:general-homogeneous-parts}
 p=p_a+p_{a+1}+\cdots+p_b,
 \qquad p_a\neq0,\quad p_b\neq0,
\end{equation}
where $p_j\in H_j$.  Define the probability measure
\begin{equation}\label{eq:general-measure}
 \dd\nu_p(z,w)=\frac{|p(z,w)|^2}{\|p\|_2^2}\,\dd m(z,w)
 \quad\text{on }\T^2
\end{equation}
and let
\[
 \mathscr H_p=P^2(\nu_p)
 =\overline{\C[z,w]}^{\,L^2(\nu_p)}.
\]
Multiplication by $z$ and $w$ on $\mathscr H_p$ will be denoted by $S_z$
and $S_w$.  They are commuting isometries.  Put
\[
 \mathscr E_z=\mathscr H_p\ominus z\mathscr H_p,
 \qquad
 \mathscr E_w=\mathscr H_p\ominus w\mathscr H_p,
\]
and denote the corresponding orthogonal projections by $Q_z$ and $Q_w$.

\begin{theorem}[Weighted model and grading rigidity]
\label{thm:general-weighted-model}
Let $\calM=[p]$ and let $p$ have the decomposition
\eqref{eq:general-homogeneous-parts}.
\begin{enumerate}[label=\textup{(\roman*)}]
\item The map initially defined on polynomials by
\begin{equation}\label{eq:general-unitary}
 U_p h=\frac{p}{\|p\|_2}h
\end{equation}
extends to a unitary $U_p:\mathscr H_p\to\calM$ satisfying
$U_pS_z=R_zU_p$ and $U_pS_w=R_wU_p$.
\item Under this unitary,
\begin{align}
 U_p^*P_zU_p&=Q_z,& U_p^*P_wU_p&=Q_w,
 \label{eq:general-defects}\\
 U_p^*X_{\calM}U_p&=[S_z^*,S_w],&
 U_p^*C_{\calM}U_p
 &=I-S_zS_z^*-S_wS_w^*+S_zS_wS_z^*S_w^*.
 \label{eq:general-operators}
\end{align}
\item If $h\in H_m$ and $g\in H_n$, then
\begin{equation}\label{eq:general-block-pairing}
 \langle h,g\rangle_{\nu_p}
 =\frac1{\|p\|_2^2}
   \sum_{\substack{a\leq r,s\leq b\\m+r=n+s}}
   \langle p_rh,p_sg\rangle_{H^2(\D^2)}.
\end{equation}
Consequently $H_m\perp H_n$ in $\mathscr H_p$ whenever
$|m-n|>b-a$.
\item The spaces $H_0,H_1,\ldots$ are mutually orthogonal in
$\mathscr H_p$ if and only if $p$ is homogeneous.
\end{enumerate}
\end{theorem}

\begin{proof}
For a polynomial $h$,
\[
 \left\|\frac p{\|p\|_2}h\right\|_{H^2(\D^2)}^2
 =\int_{\T^2}|h|^2\,\dd\nu_p=\|h\|_{\nu_p}^2.
\]
The range of the extension of \eqref{eq:general-unitary} is the closure of
$p\C[z,w]$, namely $[p]$.  This proves (i), including the intertwining
relations.  Part (ii) follows by taking orthogonal complements of the ranges
of the two intertwined isometries and then substituting the intertwining
relations into the definitions of $X_{\calM}$ and $C_{\calM}$.

For (iii), expand $p$ into its homogeneous parts.  Since $p_rh$ and $p_sg$
have total degrees $m+r$ and $n+s$, respectively,
\[
 \langle h,g\rangle_{\nu_p}
 =\frac{\langle ph,pg\rangle_{H^2(\D^2)}}{\|p\|_2^2}
 =\frac1{\|p\|_2^2}
   \sum_{m+r=n+s}\langle p_rh,p_sg\rangle,
\]
which is \eqref{eq:general-block-pairing}.  The displayed degree constraint
implies $|m-n|=|s-r|\leq b-a$ whenever a summand is non-zero.

If $a=b$, formula \eqref{eq:general-block-pairing} gives the mutual
orthogonality in (iv).  Conversely, suppose $a<b$.  Take $h=p_b\in H_b$ and
$g=p_a\in H_a$.  In \eqref{eq:general-block-pairing} the relation
$b+r=a+s$, with $a\leq r,s\leq b$, forces $r=a$ and $s=b$.  Hence
\[
 \langle p_b,p_a\rangle_{\nu_p}
 =\frac{\langle p_ap_b,p_bp_a\rangle}{\|p\|_2^2}
 =\frac{\|p_ap_b\|_2^2}{\|p\|_2^2}>0.
\]
Thus two distinct total-degree layers fail to be orthogonal.
\end{proof}

The theorem gives a concrete finite-band replacement for the scalar Toeplitz
matrices of Section~\ref{sec:toeplitz}.  To state it, write
\[
 p(z,w)=\sum_{\gamma\in A}c_\gamma z^{\gamma_1}w^{\gamma_2},
 \qquad A\subset\Nzero^2\text{ finite},
\]
and, for $\alpha,\beta\in\Nzero^2$, let
$G_p(\alpha,\beta)=\langle z^{\alpha_1}w^{\alpha_2},
z^{\beta_1}w^{\beta_2}\rangle_{\nu_p}$.

\begin{corollary}[Doubly Toeplitz moment matrix]
\label{cor:general-moment-matrix}
The moment matrix of $\nu_p$ is
\begin{equation}\label{eq:general-moment-formula}
 G_p(\alpha,\beta)
 =\frac1{\|p\|_2^2}
   \sum_{\substack{\gamma,\delta\in A\\
                    \alpha+\gamma=\beta+\delta}}
       c_\gamma\overline{c_\delta}.
\end{equation}
In particular, $G_p(\alpha,\beta)=0$ unless
$\beta-\alpha\in A-A$.  It is a doubly Toeplitz matrix with finite Fourier
bandwidth, and its total-degree block bandwidth is at most $b-a$.
\end{corollary}

\begin{proof}
Expanding the two copies of $p$ in
$\langle pz^{\alpha_1}w^{\alpha_2},
pz^{\beta_1}w^{\beta_2}\rangle/\|p\|_2^2$ gives
\eqref{eq:general-moment-formula}.  The remaining assertions follow at once
from the equality of the two exponent vectors and
Theorem~\ref{thm:general-weighted-model}(iii).
\end{proof}

\subsection{Bivariate orthogonal polynomials and the core spectrum}

The moment matrix in Corollary~\ref{cor:general-moment-matrix} is not merely
a device for computing norms.  Its orthogonal-polynomial connection matrices
determine the spectrum of the core.  The relevant object is the relative
position of the two edge spaces $\mathscr E_z$ and $\mathscr E_w$.

Define the weighted-model core operator
\begin{equation}\label{eq:general-weighted-core}
 \mathbf C_p=I-S_zS_z^*-S_wS_w^*
        +S_zS_wS_z^*S_w^*.
\end{equation}
By Theorem~\ref{thm:general-weighted-model}, it is unitarily equivalent to
$C_{[p]}$.  Introduce the edge cross-Gram operator
\begin{equation}\label{eq:Gamma-p}
 \Gamma_p=Q_w|_{\mathscr E_z}:
 \mathscr E_z\longrightarrow\mathscr E_w.
\end{equation}

\begin{theorem}[Bivariate cross-Gram model for the core]
\label{thm:general-core-Gram}
The space $\mathscr E_z$ reduces $\mathbf C_p^2$, and
\begin{equation}\label{eq:core-square-Gram}
 \mathbf C_p^2|_{\mathscr E_z}
 =Q_zQ_wQ_z|_{\mathscr E_z}
 =\Gamma_p^*\Gamma_p.
\end{equation}
If $(e_i)$ and $(f_j)$ are arbitrary orthonormal bases of
$\mathscr E_z$ and $\mathscr E_w$, respectively, then the matrix of
$\Gamma_p$ is the bivariate cross-Gram matrix
\begin{equation}\label{eq:edge-cross-Gram}
 \big(\langle e_i,f_j\rangle_{\nu_p}\big)_{j,i}.
\end{equation}

Assume, in particular, that $\mathbf C_p$ is compact.  If
$s\in(0,1)$ is a singular value of $\Gamma_p$ of multiplicity $m$, then
$s$ and $-s$ are eigenvalues of $\mathbf C_p$, each of multiplicity $m$.
The endpoint eigenspaces are
\begin{align}
 \ker(\mathbf C_p-I)
 &=\mathscr H_p\ominus(z\mathscr H_p+w\mathscr H_p),
 \label{eq:core-plus-one}\\
 \ker(\mathbf C_p+I)
 &=(z\mathscr H_p\cap w\mathscr H_p)\ominus zw\mathscr H_p.
 \label{eq:core-minus-one}
\end{align}
Consequently, the complete non-zero core spectrum is determined by
$\Gamma_p$ together with the two endpoint spaces.
\end{theorem}

\begin{proof}
Because $S_z$ and $S_w$ are commuting isometries,
\begin{equation}\label{eq:general-core-proj-difference}
 \mathbf C_p=Q_z-S_wQ_zS_w^*.
\end{equation}
The second term is the orthogonal projection onto $w\mathscr E_z$.
Moreover, $S_z^*$ and $S_w^*$ commute, and hence
\[
 S_w^*\mathscr E_z\subset\mathscr E_z.
\]
Writing $P=Q_z$ and $Q=S_wQ_zS_w^*$, the standard identity
$(P-Q)^2=P+Q-PQ-QP$ shows that $P$ commutes with $(P-Q)^2$.
Thus $\mathscr E_z$ reduces $\mathbf C_p^2$, and for $h\in\mathscr E_z$,
\begin{align*}
 \mathbf C_p^2h
 &=h-Q_zS_wQ_zS_w^*h\\
 &=h-Q_zS_wS_w^*h
 =Q_zQ_wh.
\end{align*}
This proves the first equality in \eqref{eq:core-square-Gram}.  Since
$\Gamma_p^*=Q_z|_{\mathscr E_w}$, the second follows.
Equation \eqref{eq:edge-cross-Gram} is the matrix definition of the
orthogonal projection $Q_w|_{\mathscr E_z}$.

For completeness, \eqref{eq:general-core-proj-difference} is a difference
of two orthogonal projections.  The two-projection decomposition
\cite{HalmosTwoSubspaces} gives the symmetric eigenvalue pairs in $(-1,1)$;
their squares are the eigenvalues of \eqref{eq:core-square-Gram}.  The
$+1$ eigenspace is
\[
 \mathscr E_z\cap(w\mathscr E_z)^\perp
 =\mathscr E_z\cap\mathscr E_w,
\]
which is \eqref{eq:core-plus-one}.  The $-1$ eigenspace is
\[
 z\mathscr H_p\cap(w\mathscr E_z)
 =(z\mathscr H_p\cap w\mathscr H_p)\ominus zw\mathscr H_p,
\]
which proves \eqref{eq:core-minus-one}.
\end{proof}

\begin{remark}
For the polynomial principal submodules considered here, compactness is not
an additional hypothesis: algebraic submodules of $\HH$ are
Hilbert--Schmidt by the results in \cite{Yang1999,Yang2001}; see also the
recent submodule--quotient comparison in \cite{WangZhaoZhu}.  We retained the
conditional wording in Theorem~\ref{thm:general-core-Gram} because the
operator identity \eqref{eq:core-square-Gram} holds for arbitrary weighted
principal models, including non-polynomial generators for which compactness
may fail.
\end{remark}

\begin{remark}[Geometric meaning]
The singular values of $\Gamma_p$ are the canonical correlations between the
two edge wandering spaces.  Equivalently, they are the sines of the principal
angles between $\mathscr E_z$ and $w\mathscr E_z$.  The core is the difference
of the projections onto the latter two spaces, so its positive and negative
eigenvalues record the same angles with opposite orientations.  Homogeneity
turns every total-degree edge into a line, and
\eqref{eq:core-square-Gram} then reduces exactly to the scalar Verblunsky
formula in Theorem~\ref{thm:core-spectrum}.
\end{remark}

The connection with bivariate orthogonal polynomials can also be stated in
matrix form.  Choose any ordering of the analytic monomials and apply
Gram--Schmidt in $L^2(\nu_p)$ to obtain a complete orthonormal polynomial
system $(\Phi_j)$.  Let
\begin{equation}\label{eq:bivariate-mult-matrices}
 Z_{ij}=\langle z\Phi_j,\Phi_i\rangle_{\nu_p},
 \qquad
 W_{ij}=\langle w\Phi_j,\Phi_i\rangle_{\nu_p}.
\end{equation}

\begin{corollary}[Matrix reconstruction]
\label{cor:bivariate-reconstruction}
The doubly Toeplitz moment matrix $G_p$ determines $Z$ and $W$, and
\begin{equation}\label{eq:core-bivariate-matrix}
 [\mathbf C_p]_{(\Phi_j)}
 =I-ZZ^*-WW^*+ZWZ^*W^*.
\end{equation}
Furthermore,
\[
 \mathscr E_z=\ker Z^*,\qquad \mathscr E_w=\ker W^*,
\]
and the singular values of the connection matrix between orthonormal bases of
these two kernels give the absolute values of the core eigenvalues in
$(-1,1)$.
\end{corollary}

\begin{proof}
The moment matrix determines the orthonormalization coefficients and hence
the two multiplication matrices in \eqref{eq:bivariate-mult-matrices}.
Equation \eqref{eq:core-bivariate-matrix} is
\eqref{eq:general-weighted-core} in this basis.  The two kernel identities
follow from the definitions of the defect spaces, and the last assertion is
Theorem~\ref{thm:general-core-Gram}.
\end{proof}

Lexicographic and reverse-lexicographic orthogonal vector polynomials provide
two natural finite-section bases for the edge spaces.  The recurrence and
connection matrices of Geronimo and Woerdeman
\cite{GeronimoWoerdemanBicircle} therefore contain enough information to
reconstruct \eqref{eq:edge-cross-Gram}.  What is not automatic is a canonical
ordering-independent formula which is also invariant under a change of
polynomial generator by a cyclic factor.  This distinction motivates
Problem~\ref{prob:block-model} below.  For general background on multivariate
moment matrices and orthogonalization, see \cite{DunklXu}.

\begin{example}[The first coupled blocks]\label{ex:one-plus-z-plus-w}
Take $p(z,w)=1+z+w$.  Here $\|p\|_2^2=3$ and the first moment matrix, indexed
by $1,z,w$, is
\[
 \begin{pmatrix}
  1&1/3&1/3\\
  1/3&1&1/3\\
  1/3&1/3&1
 \end{pmatrix}.
\]
Thus constants are not orthogonal to total degree one, and the two degree-one
monomials are themselves coupled.  This elementary matrix already shows why
one scalar Toeplitz matrix and one Verblunsky sequence cannot describe the
general problem.  On the other hand, its Cholesky factors begin the
bivariate orthogonalization which, through
Corollary~\ref{cor:bivariate-reconstruction}, reconstructs the core.
\end{example}

There is a second obstruction to extending the homogeneous formulas: a
principal submodule does not remember cyclic factors of its generator.

\begin{proposition}[Cyclic factors and failure of the scalar determinant]
\label{prop:cyclic-factor}
If $u\in\C[z,w]$ is cyclic in $H^2(\D^2)$, then
\begin{equation}\label{eq:cyclic-factor}
 [up]=[p]
 \qquad(0\neq p\in\C[z,w]).
\end{equation}
In particular, let $p_\lambda(z,w)=1+\lambda z$ with
$0<|\lambda|<1$.  Then $[p_\lambda]=H^2(\D^2)$ and
\begin{equation}\label{eq:stable-example-operators}
 P_zP_w=P_{\C\one},\qquad X_{[p_\lambda]}=0,
 \qquad C_{[p_\lambda]}=P_{\C\one}.
\end{equation}
Nevertheless, for $q_\lambda(\zeta)=p_\lambda(\zeta,1)$,
\begin{equation}\label{eq:stable-example-mahler}
 \frac{\mathfrak M(q_\lambda)^2}{\|q_\lambda\|_2^2}
 =\frac1{1+|\lambda|^2}<1
 =\det(I-X_{[p_\lambda]}^*X_{[p_\lambda]}).
\end{equation}
Thus the slice and Fredholm--Mahler formulas from the homogeneous case do not
extend verbatim to general polynomial generators.
\end{proposition}

\begin{proof}
Choose polynomials $h_n$ with $uh_n\to1$ in $H^2(\D^2)$.  Since
multiplication by $p$ is bounded, $puh_n\to p$.  Hence $p\in[up]$, and
therefore $[p]\subset[up]$; the reverse inclusion is immediate.

For $p_\lambda$, the reciprocal
$(1+\lambda z)^{-1}=\sum_{n\geq0}(-\lambda z)^n$ belongs to
$H^\infty(\D^2)$, so $p_\lambda$ is cyclic.  On the full Hardy module the
coordinate shifts doubly commute, their two wandering spaces are
$H^2(w)$ and $H^2(z)$, and their common part is $\C\one$.  This proves
\eqref{eq:stable-example-operators}.  Finally Jensen's formula gives
$\mathfrak M(1+\lambda\zeta)=1$, while
$\|1+\lambda\zeta\|_2^2=1+|\lambda|^2$, proving
\eqref{eq:stable-example-mahler}.
\end{proof}

The higher invariants also admit a basis-free formulation that remains
meaningful in the general model.  For $k\geq0$ define
\begin{equation}\label{eq:general-Bk}
 B_k(p)=Q_wS_z^{*k}S_w^k\big|_{\mathscr E_z}:
 \mathscr E_z\longrightarrow\mathscr E_w.
\end{equation}

\begin{proposition}[General higher-invariant operator]
\label{prop:general-Bk}
For arbitrary orthonormal bases $(e_i)$ of $\mathscr E_z$ and $(f_j)$ of
$\mathscr E_w$,
\begin{equation}\label{eq:general-sigma-k}
 \|B_k(p)\|_{\HS}^2
 =\sum_{i,j}\left|
   \langle S_w^ke_i,S_z^kf_j\rangle_{\nu_p}\right|^2,
\end{equation}
whenever either side is finite.  The value is independent of the bases and
depends only on the submodule $[p]$.  If $p$ is homogeneous, then
\eqref{eq:general-sigma-k} reduces to $\Sigma_k([p])$ in
\eqref{eq:sigma-k-def}.
\end{proposition}

\begin{proof}
For every $i,j$,
\[
 \langle B_k(p)e_i,f_j\rangle
 =\langle S_z^{*k}S_w^ke_i,f_j\rangle
 =\langle S_w^ke_i,S_z^kf_j\rangle.
\]
Parseval's identity gives \eqref{eq:general-sigma-k} and proves basis
independence.  The unitary $U_p$ from
Theorem~\ref{thm:general-weighted-model} identifies $B_k(p)$ with
$P_{\calM\ominus w\calM}R_z^{*k}R_w^k
\big|_{\calM\ominus z\calM}$, so the operator depends only on $[p]$.
In the homogeneous case use the bases from
Theorem~\ref{thm:wandering-bases}; orthogonality of distinct total-degree
layers removes all terms with unequal indices and gives
\eqref{eq:sigma-k-def}.
\end{proof}

\section{Problems and further directions}\label{sec:problems}

The homogeneous theory gives a completely solvable benchmark: scalar
Verblunsky coefficients govern the two wandering spaces, their projection
product, the cross-commutator, the core, and the higher invariants.  Theorem
\ref{thm:general-weighted-model} and Proposition~\ref{prop:cyclic-factor}
show that the general polynomial problem has two genuinely new features:
finite-band coupling between degrees and invariance under cyclic factors.
The following problems separate these difficulties into testable steps.

\begin{problem}[A canonical block model]\label{prob:block-model}
Starting from the doubly Toeplitz matrix $G_p$ in
\eqref{eq:general-moment-formula}, construct canonical recurrence parameters
that determine $Q_z$, $Q_w$, $B_k(p)$ and the core operator.  Is there a
block-CMV or finite-band unitary realization which reduces to the scalar CMV
matrix of Section~\ref{sec:cmv} when $p$ is homogeneous?  Determine which
lexicographic, reverse-lexicographic or total-degree ordering makes the
resulting parameters intrinsic to the submodule rather than to the chosen
generator.
\end{problem}

The bivariate recurrences of \cite{GeronimoWoerdemanBicircle} provide a
natural point of comparison, but here the required data must simultaneously
encode two defect spaces and be unchanged under the cyclic-factor operation
\eqref{eq:cyclic-factor}.

\begin{problem}[The sharp quotient degree constant]
\label{prob:quotient-degree}
For $D\geq1$, define
\[
 \mathfrak Q_D=
 \sup\bigl\{\|[S_z^*,S_w]\|_{\HS}^2:
 0\neq p\in\C[z,w]\text{ is homogeneous of degree }D\bigr\}.
\]
Theorems~\ref{thm:quotient-singular-values} and
\ref{thm:quotient-unbounded} give
\[
 D-o(D)\leq \mathfrak Q_D
 \leq1+\log\binom{2D}{D}.
\]
Determine the sharp asymptotics of $\mathfrak Q_D$.  In particular, is
$\mathfrak Q_D/D\to1$?  Are powers of a single toral linear factor
asymptotically extremal, and which equality or stability conditions can be
read from the Verblunsky coefficients and the endpoint weights $a_q,b_q$?
\end{problem}

This problem isolates a quantitative issue which does not arise at fixed
degree: Corollary~\ref{cor:quotient-HS} supplies a degree-dependent bound,
whereas the family $(z-w)^D$ proves that no absolute bound is possible.

\begin{problem}[Schatten thresholds and toral zero geometry]
\label{prob:schatten-general}
For $0\neq p\in\C[z,w]$, determine
\[
 \tau(p)=\inf\{r>0:C_{[p]}\in\Sr_r\},
\]
with the convention that the infimum of the empty set is $\infty$.  Express
$\tau(p)$ in terms of the equivalence class of $p$ under insertion or removal
of cyclic polynomial factors, the geometry of $Z(p)\cap\T^2$, and local
intersection or contact multiplicities.  When $0<\tau(p)<\infty$, determine
the possible endpoint membership in the weak Schatten class
$\Sr_{\tau(p),\infty}$.  As a first case, settle the problem for polynomials
whose toral zero set is finite.
\end{problem}

For homogeneous $p$, Theorem~\ref{thm:intro-dictionary} turns this question
into the membership of $(\alpha_n)$ in $\ell^r$.  Problem
\ref{prob:schatten-general} asks for the geometric substitute for that scalar
criterion when the degree blocks interact.

\begin{problem}[A cyclically invariant determinant]\label{prob:det-general}
Assume that $X_{[p]}^*X_{[p]}$ is trace class.  Find an explicit quantity
$\mathcal D([p])$, computable from $p$ but invariant under $p\mapsto up$ for
every cyclic polynomial $u$, such that
\[
 \det(I-X_{[p]}^*X_{[p]})=\mathcal D([p]).
\]
Can $\mathcal D([p])$ be represented as a limit of determinants of finite
sections of $G_p$?  In the homogeneous case it must reduce to
$\mathfrak M(p(\cdot,1))^2/\|p(\cdot,1)\|_2^2$; Proposition
\ref{prop:cyclic-factor} shows that this ratio itself is not the answer in
general.
\end{problem}

The finite-section formulation links this question to multivariable
Fej\'er--Riesz and positive-extension problems; compare
\cite{GeronimoWoerdemanAnnals}.  The essential extra requirement is
cyclic-factor invariance.

\begin{problem}[The quadratic phase diagram]\label{prob:quadratic-family}
For
\[
 p_\theta(z,w)=z^2-2\cos\theta\,zw+w^2,
 \qquad 0<\theta<\pi,
\]
determine the sign of
\[
 \Delta(\theta)=\Sigma_4([p_\theta])-\Sigma_3([p_\theta])
\]
on the whole parameter interval.  Determine all zeros of $\Delta$, and in
particular decide the maximal interval containing $\pi/6$ on which
$\Delta(\theta)>0$.  A satisfactory solution should replace pointwise
computer algebra by a parameter-uniform Schur or CMV certificate.
\end{problem}

\begin{problem}[Multiple toral factors and general higher invariants]
\label{prob:multiple-fh}
For
\[
 p(z,w)=\prod_{j=1}^N(z-e^{i\theta_j}w)^{m_j},
\]
derive asymptotics of the Verblunsky coefficients which retain the relative
phases $\theta_j-\theta_\ell$, and use them to determine the large-$k$
behaviour of $\Sigma_k([p])$.  For a non-homogeneous polynomial, determine
when the intrinsic operator $B_k(p)$ in \eqref{eq:general-Bk} is
Hilbert--Schmidt and whether $\|B_k(p)\|_{\HS}^2$ has a limit or an
asymptotic expansion as $k\to\infty$.
\end{problem}

Together, Problems~\ref{prob:block-model}--\ref{prob:multiple-fh} ask which
parts of the homogeneous scalar OPUC bridge survive in the finite-band model,
and how sharply the quotient estimates depend on degree, separating the
algebraic, Schatten, determinant and high-order aspects.

\vspace{0.3cm}

\noindent\textbf{AI disclosure.}
The research questions and main mathematical ideas in this article originated with the authors. ChatGPT 5.6 Sol (OpenAI) was then used to explore and verify these ideas through symbolic and numerical calculations. In particular, it found and verified the quadratic counterexample in Section 8. It also assisted with the exposition and language editing. The authors reviewed all AI-assisted material and take full responsibility for the results, proofs, references, and final text.

\vspace{0.3cm}

\noindent\textbf{Acknowledgments.}
The authors sincerely thank Professor Penghui Wang of Shandong University and Professor Rongwei Yang of the University at Albany for their encouragement and valuable discussions on this topic.

\appendix

\section{Exact certificates for the quadratic example}\label{app:certificate}

\subsection{Schur iterates}\label{app:schur-certificate}

Let $\eps_m=(-1)^{m+1}$.  The Schur iterates in
Lemma~\ref{lem:six-step} have the form
\begin{equation}\label{eq:schur-table-form}
 f_{6m+r}(\zeta)=
 \frac{A_r+B_r\zeta}{1+C_r\zeta+D_r\zeta^2}.
\end{equation}
The coefficients are as follows; all entries depend on $m$.

\begingroup
\hfuzz=30pt
\vfuzz=40pt
\renewcommand{\arraystretch}{1.10}
\small
\begin{center}
\resizebox{0.82\textwidth}{!}{%
\begin{tabular}{c>{\raggedright\arraybackslash}p{0.185\textwidth}
                  >{\raggedright\arraybackslash}p{0.215\textwidth}
                  @{\hspace{2em}}
                  >{\raggedright\arraybackslash}p{0.22\textwidth}}
\toprule
$r$ & $A_r$ & $B_r$ & $(C_r,D_r)$\\
\midrule
$0$ & $\displaystyle\frac{2\eps_m\sqrt3}{36m(m+1)+5}$
& $\displaystyle\frac{\eps_m(6m-1)}{36m(m+1)+5}$
& $\displaystyle C_r=-\frac{\sqrt3(36m^2+30m+2)}{36m(m+1)+5}$,
  $\displaystyle D_r=\frac{36m^2+24m+1}{36m(m+1)+5}$\\
$1$ & $\displaystyle\frac{\eps_m(6m+7)}{36m^2+48m+13}$
& $\displaystyle-\frac{2\eps_m\sqrt3}{36m^2+48m+13}$
& $\displaystyle C_r=-\frac{\sqrt3(36m^2+42m+8)}{36m^2+48m+13}$,
  $\displaystyle D_r=\frac{36m^2+36m+5}{36m^2+48m+13}$\\
$2$ & $\displaystyle\frac{\eps_m\sqrt3}{6m+4}$
& $\displaystyle-\frac{\eps_m(6m+7)}{(6m+4)(6m+6)}$
& $\displaystyle C_r=-\frac{(6m+3)\sqrt3}{6m+4}$,
  $\displaystyle D_r=\frac{36m^2+48m+13}{(6m+4)(6m+6)}$\\
$3$ & $\displaystyle\frac{\eps_m}{3m+3}$
& $\displaystyle-\frac{\eps_m\sqrt3}{6(m+1)}$
& $\displaystyle C_r=-\frac{(6m+5)\sqrt3}{6(m+1)}$,
  $\displaystyle D_r=\frac{3m+2}{3m+3}$\\
$4$ & $\displaystyle\frac{\eps_m\sqrt3}{6m+8}$
& $\displaystyle-\frac{\eps_m}{3m+4}$
& $\displaystyle C_r=-\frac{(6m+7)\sqrt3}{6m+8}$,
  $\displaystyle D_r=\frac{3m+3}{3m+4}$\\
$5$ & $\displaystyle\frac{\eps_m(6m+5)}{\Delta_m}$
& $\displaystyle-\frac{6\eps_m(m+1)\sqrt3}{\Delta_m}$
& $\displaystyle C_r=-\frac{\sqrt3(36m^2+90m+54)}{\Delta_m}$,
  $\displaystyle D_r=\frac{36m^2+84m+48}{\Delta_m}$\\
\bottomrule
\end{tabular}}
\end{center}
\endgroup

Here the abbreviation in the last row is
$\Delta_m=36m^2+96m+61$.

Inserting each row into \eqref{eq:Schur-algorithm} gives the next row after
cancelling a common scalar; the last row gives the first row with $m$ replaced
by $m+1$.  Thus the table is also an exact rational-function certificate for
Lemma~\ref{lem:six-step}.

\subsection{Alternating-CMV blocks}\label{app:cmv-certificate}

For $m\geq1$, exact multiplication of the alternating CMV blocks and
substitution of \eqref{eq:a0}--\eqref{eq:a5} gives the following values.
The common sign $\eps_m=(-1)^{m+1}$ disappears after taking absolute
squares, but is retained here so that the unsquared formulas are exact.

\begingroup
\small
\begin{align*}
F_{3,0}(m)&=\eps_m\frac{6\sqrt3m}{36m^2+24m+1},\\
F_{4,0}(m)&=\eps_m\frac{12(36m^3+24m^2+m+1)}
 {(6m+1)(6m+5)(36m^2+24m+1)},\\[0.4ex]
F_{3,1}(m)&=\eps_m\frac{12(2m+1)(18m^2+15m-1)}
 {(6m+1)(6m+5)(36m^2+48m+13)},\\
F_{4,1}(m)&=\eps_m\frac{6\sqrt3m}{36m^2+48m+13},\\[0.4ex]
F_{3,2}(m)&=\eps_m\frac{\sqrt3(108m^3+180m^2+66m-7)}
 {6(m+1)(3m+2)(36m^2+48m+13)},\\
F_{4,2}(m)&=\eps_m\frac{216m^3+252m^2-6m-49}
 {12(m+1)(3m+2)(36m^2+48m+13)},\\[0.4ex]
F_{3,3}(m)&=\eps_m\frac{6m-1}{12(m+1)(3m+2)},\\
F_{4,3}(m)&=-\eps_m\frac{\sqrt3}{6(m+1)(3m+2)},\\[0.4ex]
F_{3,4}(m)&=-\eps_m\frac{\sqrt3}{6(m+1)(3m+4)},\\
F_{4,4}(m)&=-\eps_m\frac{6m+7}{12(m+1)(3m+4)},\\[0.4ex]
F_{3,5}(m)&=-\eps_m\frac{216m^3+828m^2+1002m+391}
 {12(m+1)(3m+4)(36m^2+96m+61)},\\
F_{4,5}(m)&=-\eps_m\frac{\sqrt3(108m^3+360m^2+390m+139)}
 {6(m+1)(3m+4)(36m^2+96m+61)}.
\end{align*}
\endgroup

At the boundary block $m=0$, direct calculation gives
\begin{align}
 (F_3(0),\ldots,F_3(5))
 &=\left(0,\frac{12}{65},\frac{7\sqrt3}{156},\frac1{24},
          \frac{\sqrt3}{24},\frac{391}{2928}\right),\label{eq:F3-initial}\\
 (F_4(0),\ldots,F_4(5))
 &=\left(0,0,\frac{49}{312},\frac{\sqrt3}{12},
          \frac7{48},\frac{139\sqrt3}{1464}\right).
 \label{eq:F4-initial}
\end{align}

Squaring the table entries and simplifying gives
$B_m=-18P(m)/Q(m)$, where
\begin{align}
Q(m)={}&(3m+2)^2(3m+4)^2(6m+1)^2(6m+5)^2\notag\\
&\times(36m^2+24m+1)^2(36m^2+48m+13)^2
       (36m^2+96m+61)^2,\label{eq:Q-poly}
\end{align}
and
\begingroup\small
\begin{equation}\label{eq:P-poly}
\begin{aligned}
P(m)={}&2115832430592m^{17}+25037350428672m^{16}\\
&+135648368050176m^{15}+445363135598592m^{14}\\
&+988003952815104m^{13}+1562231786411520m^{12}\\
&+1807981351520256m^{11}+1545118573697280m^{10}\\
&+967871622362688m^9+431021958821280m^8\\
&+125478742327488m^7+17101374115824m^6\\
&-2736799083324m^5-1896379326126m^4\\
&-440148563712m^3-60242904189m^2\\
&-5688506082m-323586661.
\end{aligned}
\end{equation}
\endgroup
Finally, the numerator obtained after clearing denominators in
$B_m+1/(6m^3)$ is
\begingroup\small
\begin{equation}\label{eq:R-poly}
\begin{aligned}
R(m)={}&799784658763776m^{19}+10359115580178432m^{18}\\
&+62275295929614336m^{17}+230702433860136960m^{16}\\
&+589662544930707456m^{15}+1103159941756194816m^{14}\\
&+1563846250906377216m^{13}+1715574697139817216m^{12}\\
&+1474492477269211776m^{11}+998835839669371392m^{10}\\
&+533684659694031936m^9+223967804186097936m^8\\
&+73126009637189640m^7+18284181806875536m^6\\
&+3417662671887924m^5+460740759213393m^4\\
&+42451446056544m^3+2459278339956m^2\\
&+77909014560m+1006158400.
\end{aligned}
\end{equation}
\endgroup
Every coefficient in \eqref{eq:R-poly} is positive, which is the exact
certificate used in Lemma~\ref{lem:block-tail}.

\end{document}